%% file: main.tex
\documentclass[12pt]{article}

\usepackage[T1]{fontenc} \usepackage[english]{babel}
\usepackage{latexsym, amsmath, amssymb, amsthm, color, xcolor,graphicx}
\usepackage{enumitem}
\usepackage{authblk}
\usepackage{titlefoot}

\begin{document}

\newcommand{\corr}[2]{{\color{red}\sout{#1}} {\color{blue} {#2} \color{black}}}
\newcommand{\primozcomment}[1] {\color{magenta} ***  #1 *** \color{black}}
\newcommand{\roggicomment}[1] {\color{green} ***  #1 *** \color{black}}
\newcommand{\wcomment}[1] {\color{orange} ***  #1 *** \color{black}}

\newtheorem{theorem}{Theorem}[section]
\newtheorem{lemma}[theorem]{Lemma}
\newtheorem{corollary}[theorem]{Corollary}
\newtheorem{question}[theorem]{Question}
\newtheorem{proposition}[theorem]{Proposition}
\theoremstyle{definition}
\newtheorem{definition}[theorem]{Definition}
\newtheorem{example}[theorem]{Example}
\newtheorem*{remark}{Remark}

\newtheorem{thmx}{Theorem}
\renewcommand{\thethmx}{\Alph{thmx}} 

\newcommand{\Z}{\mbox{${\bf Z}$}}
\newcommand{\Q}{\mbox{${\bf Q}$}}
\newcommand{\N}{\mbox{${\bf N}$}}
\newcommand{\elt}{\mbox{$v \in G$}}
\newcommand{\elg}{\mbox{$\phi \in \Gamma$}}
\newcommand{\gstab}{\mbox{$G_\alpha$}}
\newcommand{\caa}{\mbox{$C_1(g)$}}
\newcommand{\cia}{\mbox{$C_i(g)$}}
\newcommand{\ima}{\mbox{${\rm Im}$}}
\newcommand{\de}{\mbox{${\rm deg}$}}
\newcommand{\dist}{\mbox{${\rm dist}$}}
\newcommand{\wre}{\mbox{${\rm Wr\,}$}}
\newcommand{\res}{\mbox{$\mid$}}
\newcommand{\nin}{\mbox{$\ \not\in\ $}}
\newcommand{\fix}{\mbox{${\rm fix}$}}
\newcommand{\pr}{\mbox{${\rm pr}$}}
\newcommand{\sym}{\mbox{${\rm Sym}\;$}}
\newcommand{\aut}{\mbox{${\rm Aut}$}}
\newcommand{\auttn}[0]{\mbox{${\rm Aut}\;T_n$}}
\newcommand{\autx}[0]{\mbox{${\rm Aut}\;X$}}
\newcommand{\autga}[0]{\mbox{${\rm Aut}(\Gamma)$}}
\newcommand{\autde}[0]{\mbox{${\rm Aut}(\Delta)$}}
\newcommand{\out}{\mbox{${\rm out}$}}
\newcommand{\inn}{\mbox{${\rm in}$}}
\newcommand{\Al}{\mbox{${\rm Al}$}}
\newcommand{\Pl}{\mbox{${\rm Pl}$}}
\newcommand{\Fix}{\mbox{${\rm Fix}$}}
\newcommand{\rightquotient}{{\setminus}}

\newcommand{\ZZ}{\mathbb{Z}}
\newcommand{\RR}{\mathbb{R}}
\newcommand{\QQ}{\mathbb{Q}}
\newcommand{\NN}{\mathbb{N}}

\newcommand{\vK}{\vec{K}}
\newcommand{\vY}{\vec{Y}}
\newcommand{\vZZ}{\vec{\mathbb{Z}}}

\renewcommand{\div}{{\hbox { \rm div }}}
\renewcommand{\mod}{{\hbox { \rm mod }}}
\newcommand{\id}{\hbox{\rm id}}
\newcommand{\Aut}{\mathrm{Aut}}
\newcommand{\tG}{{\widetilde{G}}}
\newcommand{\tH}{{\widetilde{H}}}
\newcommand{\tHphi}{{\widetilde{H}}_\varphi}
\newcommand{\CDC}{\mathop{{\rm CDC}}}
\newcommand{\CDHC}{\mathop{{\rm CDHC}}}
\newcommand{\g}{\mathbf g}
\newcommand{\h}{\mathbf h}
\newcommand{\V}{\mathrm V}
\newcommand{\VX}{{\mathrm V}X}
\newcommand{\E}{\mathrm E}
\newcommand{\A}{\mathrm A}
\newcommand{\R}{\mathrm R}
\newcommand{\TFaut}{{\mathop{{\rm TF}}}}
\newcommand{\Sym}{\mathrm{Sym}}
\newcommand{\Star}{{\mathop{{\rm Star}}}}
\newcommand{\md}{\mbox{${\rm md}$}}
\newcommand{\im}{\mbox{${\rm im}$}}

\newcommand{\tdlc}{totally disconnected, locally compact }
\newcommand{\cgtdlc}{compactly generated, totally disconnected, locally compact }

\title{Cayley--Abels graphs and invariants of totally disconnected, locally compact groups}
 
\author{Arnbj\"org Soff\'ia \'Arnad\'ottir\footnote{Department of Combinatorics and Optimization, University of Waterloo,
200 University Avenue West, Waterloo, ON, Canada  N2L 3G1.  e-mail:  {\tt soffia.arnadottir@uwaterloo.ca}
},
Waltraud Lederle\footnote{IRMP, Universit\'e Catholique de Louvain, 
Chemin du Cyclotron 2, 1348 Louvain-la-Neuve, Belgium.  e-mail:  {\tt waltraud.lederle@uclouvain.be}}\\
and R\"ognvaldur G. M\"oller\footnote{Science Institute, University of Iceland, IS-107 Reykjav\'ik, Iceland.  e-mail: {\tt roggi@raunvis.hi.is}}}


\maketitle

\unmarkedfntext{The second named author was supported by the Early Postdoc.Mobility scholarship number 175106 from the Swiss National Science Foundation. Part of this work was done when she was visiting The University of Newcastle with the International Visitor Program of the Sydney Mathematical Research Institute.}

\begin{abstract}
A connected, locally finite graph $\Gamma$ is a Cayley--Abels graph for a totally disconnected, locally compact group $G$ if $G$ acts vertex-transitively with compact, open vertex stabilizers on $\Gamma$. Define the minimal degree of $G$ as the minimal degree of a Cayley--Abels graph of $G$.  We relate the minimal degree in various ways to the modular function, the scale function and the structure of compact open subgroups. As an application, we prove that if $T_d$ denotes the $d$-regular tree, then the minimal degree of $\aut(T_d)$ is $d$ for all $d\geq 2$.
\end{abstract}


\section*{Introduction}

Let $G$ be a totally disconnected, locally compact group.  A locally finite, connected graph $\Gamma$
on which $G$ acts vertex-transitively with compact, open vertex stabilizers is called a Cayley--Abels graph for $G$. It was introduced by Abels \cite{Abels1974} in the context of Specker compactifications of locally compact groups. A totally disconnected, locally finite group has a Cayley--Abels graph if and only if it is compactly generated. The relation between $G$ and a Cayley--Abels graph $\Gamma$ is in many ways similar to the relation between a finitely generated group and its Cayley graph with respect to a finite generating set, see \cite{KronMoller2008}. In this work we are interested in the minimal degree $\md(G)$ of a Cayley--Abels graph for $G$.   The following general questions seem natural starting points for an investigation of this invariant:
\begin{enumerate}
\item What does the minimal degree tell us about the group?
\item How does the minimal degree relate to other properties of the group?
\item What is the minimal degree for some interesting groups?
\end{enumerate}

In this work the authors tackle all of these questions, at least to some extent.  

For small values of $d$, we can give a characterization - the first two parts of the following theorem are obvious facts, the third is a part of Theorem \ref{T2ends}.

\begin{thmx}
 {\em Let $G$ be a compactly generated, totally disconnected, locally compact group.
 Then $\md(G)$ is
 \begin{enumerate}
     \item equal to $0$ if and only if $G$ is compact.
     \item never equal to $1$.
     \item equal to $2$ if and only if $G$ has a compact, open, normal subgroup $K$ such that $G/K$ is isomorphic to $\ZZ$ or the infinite dihedral group $D_\infty$.
 \end{enumerate}
 }
\end{thmx}

 The characterization in the third part resembles the characterization of finitely generated groups with two ends. 
In a companion paper to this work, \cite{ArnadottirLederleMoller2020}, vertex-transitive group actions on cubic graphs with infinite vertex stabilizers are studied.  These results have consequences for groups with minimal degree $3$, e.g.~if $G$ is a compactly generated, totally disconnected, locally compact group with minimal degree 3 and does not have a compact, open, normal subgroup then $G$ is not uni\-scalar (i.e. there exists an element in $G$ that does not normalize a compact, open subgroup of $G$). 

In Section \ref{sec:discrete} we elaborate on groups having a compact, open, normal subgroup. We prove that this property can be detected from some minimal degree Cayley--Abels graph.

\begin{thmx}[Corollary \ref{cor:nearly discrete implies discrete action on a minimal CA graph}]
{\em Let $G$ be a compactly generated, totally disconnected, locally compact group.
 If $G$ has a compact, open, normal subgroup, then $G$ acts on some minimal degree Cayley--Abels graph with compact, open, normal kernel.}
\end{thmx}

Our investigation of the second question starts with a study of the modular function and its relationship to the action of a group on a Cayley--Abels graph, see Section~\ref{sec:modular}. 
Many of the results in Section \ref{sec:modular} are contained in the first authors Master's thesis, \cite{Arnadottir2016}.
We present a method to \lq\lq read\rq\rq\ the values of the modular function of the graph and generalize a theorem by Praeger to higher dimensions.
Further, we get a bound on the minimal degree based on the modular function of which the following is a special case.

\begin{thmx}[Theorem \ref{C-Adigraphs}]
\label{thm:image modular function cyclic}
{\em Let $G$ be a non-compact, compactly generated, totally disconnected, locally compact group. Assume the image of the modular function is generated by $p/q$, where $p$ and $q$ are co-prime, positive integers. Then $\md(G) \geq p+q$.}
\end{thmx}

Next is the relationship of the minimal degree with the local structure of the group (i.e.\ the structure of compact, open subgroups),  see Section~\ref{sec:composition}. 
We introduce a new local invariant, the {\em local simple content}, that is inspired by the local prime content introduced by Gl\"ockner \cite{Glockner2006}.
It is the set of all (isomorphism classes of) finite, simple groups that appear as a composition factor of every compact, open subgroup.  To treat this concept, a Jordan--H\"older Theorem for profinite groups is needed (Theorem~\ref{thm:JHprofinite}).

\begin{thmx}[Theorem \ref{thm:localsimplecontent}]
\label{thm:introduct lsc}
{\em Let $G$ be a compactly generated, totally disconnected, locally compact group. Assume $G$ does not have any compact, normal subgroup.
 The local simple content of $G$ is finite.
 Moreover, $\md(G)$ is strictly bigger than the smallest $n$ such that the symmetric group $S_n$ contains every element of the local simple content as subquotient.}
\end{thmx}

This is a refinement of a theorem about the local prime content by Caprace, Reid and Willis from \cite{CapraceReidWillis2017a}.
We relate their result to the values of the scale function: If $p$ is the largest prime that occurs as a factor in any value of the scale function then $\md(G)\geq p+1$, see Corollary \ref{cor:prime factor of scale function md}; this holds in particular for non-uniscalar $p$-adic Lie groups (Corollary \ref{cor:p adic Lie group md}).

The archetypal totally disconnected, locally compact groups are  automorphism groups of regular trees.
Theorem \ref{thm:regtree} and Theorem \ref{thm:biregtree} are simple applications of Theorem \ref{thm:introduct lsc} and Theorem \ref{thm:image modular function cyclic}, except for $\aut(T_5)$.
The snag is that the alternating group $A_4$ is not simple.  Proving that $\md(\aut(T_5))=5$ is  tricky and is done in Section~\ref{sec:5-valent}.

\begin{thmx}[Corollary \ref{cor:regular tree}, Corollary \ref{cor:end stabilizer in tree}, Example \ref{expl:tree}, Theorem \ref{T5-regular}]
\label{thm:regtree}
{\em Let $T_d$ be a regular tree of degree $d \geq 2$.
Let $\omega$ be an end of $T_d$.
\begin{enumerate}
    \item The tree $T_d$ is a minimal degree Cayley--Abels graph for the automorphism group $\aut(T_d)$ and the end stabilizer $\aut(T_d)_\omega$.   In particular $\md(\aut(T_d))=d$.  
    \item Let $\aut(T_d)^+$ be the index-$2$ subgroup of $\aut(T_d)$. Then $\md(\aut(T_d)^+) \leq 2d+2$ and $\md(\aut(T_d)_\omega)) = (d-1)^2+1$.
\end{enumerate}    }
\end{thmx}

More generally, for bi-regular trees we prove the following.

\begin{thmx}[Example \ref{expl:tree}, Corollary \ref{cor:end stabilizer in tree}]
\label{thm:biregtree}
{\em Let $T_{d,d'}$ be a bi-regular tree of degrees $d> d' \geq 2$. Let $\omega$ be an end of $T_{d,d'}$.
    Then the automorphism group $\aut(T_{d,d'})$ satisfies $\md(\aut(T_{d,d'})) \leq d+d'+2$ and the end stabilizer $\aut(T_{d,d'})_\omega$ satisfies $\md(\aut(T_{d,d'})_\omega)=(d-1)(d'-1)+1$.}
\end{thmx}

The authors conjecture that $\aut(T_{d,d'})=d+d'+2$.

\bigskip

{\bf Acknowledgements.}  The seeds for this project were sown in 2014 during a visit of George Willis to the University of Iceland when he asked the third author about the minimal degree of a Cayley--Abels graph for a totally disconnected, locally compact group. Conversely, part of this work was completed when the second author was visiting George Willis at The University of Newcastle in spring 2020. The second and third authors want to thank Pierre-Emmanuel Caprace for helpful conversations.

\section{Notation and preliminaries}

\subsection{Graphs}

A graph $\Gamma$ (undirected) is defined as a pair $(\V\Gamma, \E\Gamma)$, where $\V\Gamma$ is the set of {\em vertices} and $\E\Gamma$ is a set of two element subsets of $\V\Gamma$, whose elements we call the {\em edges} of $\Gamma$.  Our graphs thus neither have loops nor multiple edges.  The set of {\em arcs} of $\Gamma$, denoted by $\A\Gamma$, is the set of all ordered pairs $(\alpha, \beta)$ such that $\{\alpha, \beta\}\in \E\Gamma$.  
Two vertices $\alpha$ and $\beta$ are said to be {\em adjacent}, or {\em neighbours}, if $\{\alpha, \beta\}$ is an edge.  The set of neighbours of a vertex $\alpha$ is denoted $\Gamma(\alpha)$ and the {\em degree} of $\alpha$ is the cardinality of $\Gamma(\alpha)$.  A graph is {\em regular} if all vertices have the same degree $d$, and then we say that $d$ is the degree of the graph.
 A graph is {\em locally finite} if every vertex has finite degree.  

 We also consider digraphs (directed graphs). A \emph{digraph} consists of a vertex set $\V\Gamma$ and a subset  $\A\Gamma \subseteq \V\Gamma \times \V\Gamma$ that does not intersect the diagonal.  Elements of $\V\Gamma$ are called vertices and elements of $\A\Gamma$ are called {\em arcs}.  
The \emph{underlying undirected graph} of a digraph $\Gamma$ has the same vertex set as $\Gamma$ and the set of edges is the set of all pairs $\{\alpha, \beta\}$ where $(\alpha, \beta)$ or $(\beta, \alpha)$ is an arc in $\Gamma$.   
If $\alpha\in \V\Gamma$ in a digraph $\Gamma$ 
we define the sets of {\em in-} and {\em out-neighbours} as
$\inn(\alpha)=\{\beta\in \V\Gamma\mid (\beta,\alpha)\in \A\Gamma\}$ and 
$\out(\alpha)=\{\beta\in \V\Gamma\mid (\alpha,\beta)\in \A\Gamma\}$, respectively.
The cardinality of $\inn(\alpha)$ is the {\em in-degree} of $\alpha$ and the cardinality of $\out(\alpha)$ is the {\em out-degree} of $\alpha$.   A digraph is {\em regular} if any two vertices have the same in-degree and also the same out-degree.  

For an integer $s \geq 0$ an \emph{$s$-arc} in $\Gamma$ (here $\Gamma$ can be an undirected graph or a digraph) is an $(s+1)$-tuple $(\alpha_0,\dots,\alpha_s)$ of vertices such that for every $0 \leq i \leq s-1$ the ordered pair $(\alpha_{i},\alpha_{i+1})$ is an arc in $\Gamma$, and $\alpha_{i-1} \neq \alpha_{i+1}$ for all $1 \leq i \leq s-1$.   Infinite arcs come in three different shapes.  There are  1-way infinite arcs
$(\alpha_0, \alpha_1, \ldots)$ and there are 2-way infinite arcs $(\ldots, \alpha_{-1}, \alpha_0, \alpha_1, \ldots)$.  In all cases we insist that $(\alpha_{i},\alpha_{i+1})$ is an arc in $\Gamma$, and $\alpha_{i-1} \neq \alpha_{i+1}$ for all $i$.

A graph $\Gamma$ is said to be {\em connected} if for every pair of vertices $\alpha$ and $\beta$ in $\Gamma$ there exists an $s$-arc $(\alpha_0,\dots,\alpha_s)$ with $\alpha = \alpha_0$ and $\beta = \alpha_s$.  The smallest possible $s$ is the {\em distance} between $\alpha$ and $\beta$ and is denoted with $d_\Gamma(\alpha, \beta)$.  

\subsection{Group actions}

Let $\Sym(\Omega)$ denote the group of all permutations of the set $\Omega$.  By a {\em permutation group} on $\Omega$ we mean a subgroup of $\Sym(\Omega)$.  
An {\em action} of a group $G$ on a set $\Omega$ is defined as a homomorphism $\pi: G\to \Sym(\Omega)$.  We write our group action on the right so that if $\alpha\in\Omega$ and $g\in G$ then $\alpha g$ denotes the image of $\alpha$ under the permutation $\pi(g)$.   
The kernel of the action (i.e.~the kernel of the homomorphism $\pi$) is equal to the set $K=\{g\in G\mid \alpha g=\alpha, \mbox{ for all }\alpha\in \Omega\}$.     If $K=\{1\}$, we say that the action is {\em faithful}.  Set $G^\Omega=G/K$.  The homomomorphism $\pi$ induces an injective homomorphism $G^\Omega\to \Sym(\Omega)$ giving an action of $G^\Omega$ on $\Omega$.  This action is clearly faithful and thus we can think of $G^\Omega$ as a permutation group on $\Omega$.  We call $G^\Omega$ {\em the permutation group on $\Omega$ induced by the action of $G$}.    If the action $\pi$ is faithful, then $G^\Omega=G$ and we can think of $G$ itself as a permutation group.

The {\em stabilizer} in $G$ of $\alpha\in \Omega$ is the subgroup $G_\alpha=\{g\in G\mid \alpha g=\alpha\}$.   For a set $A\subseteq \Omega$ the {\em pointwise stabilizer} of $A$ is the subgroup 
$G_{(A)}=\{g\in G\mid \alpha g=\alpha\mbox{ for all }\alpha \in A\}$.   

The {\em orbit} of a point $\alpha \in \Omega$ under $G$ is the set $\alpha G=\{\alpha g\mid g \in G\}$ and for $S\subseteq G$ we define the $S$-orbit of $\alpha$ as the set $\alpha S=\{\alpha g\mid g \in S\}$.  The orbits of $G_\alpha$ are called {\em suborbits} of $G$.   
An action is said to be {\em transitive} if for any two points $\alpha, \beta\in \Omega$ there is an element $g\in G$ such that $\alpha g=\beta$, i.e. every orbit under $G$ is all of $\Omega$.    A permutation group $G$ on a set $\Omega$  is called \emph{free} (or \emph{semi-regular}) if $G_\alpha=\{1\}$ for all points $\alpha\in \Omega$ and \emph{regular} if it is free and transitive.

Suppose $\Gamma$ and $\Delta$ are graphs (or digraphs).
  A \emph{graph morphism}  $\varphi: \Gamma\to \Delta$ is a map $\varphi \colon \V\Gamma \to \V\Delta$ such that if $(\alpha, \beta) \in \A\Gamma$ then $(\varphi(\alpha),\varphi(\beta)) \in \A\Delta$.
  If a graph morphism $\varphi:\Gamma\to \Gamma$ is bijective and $\varphi$ induces a bijective map $\A\Gamma\to \A\Gamma$, then $\varphi$ is an \emph{automorphism} of $\Gamma$.  The set of all automorphisms of $\Gamma$ is a group, the {\em automorphism group of $\Gamma$},  denoted $\Aut(\Gamma)$.   We think of $\aut(\Gamma)$ as a permutation group on the vertex set of $\Gamma$, and the action on the vertex set induces actions on the set of edges and the set of arcs.  
  When we say that \lq\lq a group $G$ acts on a graph $\Gamma$\rq\rq\ we always mean an action by automorphisms.   Such an action is described by a homomorphism $G\to\aut(\Gamma)$.   The permutation group on the vertex set induced by $G$ will be denoted by $G^\Gamma$.

 A graph or a digraph $\Gamma$ is {\em vertex-transitive} if the automorphism group acts transitively on the vertex set.  We say that $\Gamma$ is {\em edge-transitive} ({\em arc-transitive}, $s$-{\em arc-transitive})  if the automorphism group acts transitively on the edge set (the arc set, the set of $s$-arcs).   When the automorphism group is $s$-arc-transitive for all $s$,  then $\Gamma$ is said to be {\em highly-arc-transitive}.
 

Consider now a group $G$ that acts vertex-transitively on a graph $\Gamma$.   Let $\alpha \in \V\Gamma$.
The stabilizer $G_\alpha$ clearly leaves the set $\Gamma(\alpha)$ invariant and thus induces an action on it. The kernel of this action is $G_{(\Gamma(\alpha))} \cap G_\alpha$ and the quotient $G_\alpha^{\Gamma(\alpha)}=G_\alpha / (G_{(\Gamma(\alpha))} \cap G_\alpha)$ embeds as a subgroup into $\Sym(\Gamma(\alpha))$.
Let now $\alpha'$ be another vertex of $\Gamma$. By assumption there exists $g \in G$ with $\alpha g = \alpha'$, and $G_{\alpha'}=g^{-1} G_\alpha g$ acts on $\Gamma(\alpha')$. The actions of  $G_\alpha$ on $\Gamma(\alpha)$ and $G_{\alpha'}$ on $\Gamma(\alpha')=\Gamma(\alpha)g$ are conjugate  via $g$ and hence isomorphic. Thus, the following is independent of the choice of $\alpha$.


\begin{definition}
  Let $\Gamma$ be a graph of degree $d$ on which a group $G$ acts vertex-transitively.
  Let $\alpha \in \V\Gamma$.
  The \emph{local action} of $G$ on $\Gamma$ is the conjugacy class of the group $G_{\alpha}/(G_{(\Gamma(\alpha))} \cap G_\alpha)$, seen as a subgroup of the symmetric group $S_d$.
\end{definition}

Given a group $G$ acting on a set $\Omega$ we say that an equivalence relation on $\Omega$ is a {\em $G$-congruence} if $\alpha g$ is equivalent to $\beta g$ if and only if $\alpha$ is equivalent to $\beta$.   The orbits of a normal subgroup $N\trianglelefteq G$ form the classes of a $G$-congruence.  

When $\sim$ denotes an equivalence relation on the vertex set of a graph (or a digraph) $\Gamma$ we can form the {\em quotient graph} ({\em quotient digraph}) $\Gamma/\sim$ which has the set of $\sim$-classes as a vertex set and if $A$ and $B$ are two $\sim$-classes then $(A,B)$ is an arc in $\Gamma/\sim$ if and only if there is a vertex $\alpha\in A$ and a vertex $\beta\in B$ such that $(\alpha,\beta)$ is an arc in $\Gamma$.  
If $H$ is a subgroup of $\autga$ then $\Gamma/H$ denotes the quotient graph of $\Gamma$ with respect to the equivalence relation whose classes are the $H$-orbits on the vertex set.  If $G$ acts on $\Gamma$ and $\sim$ is a $G$-congruence then $G$ has a natural action on the set of $\sim$-classes that gives an action on the quotient digraph $\Gamma/\sim$ by automorphisms.

\subsection{The permutation topology}

When $G$ is a group acting on a set $\Omega$, e.g.~the automorphism group of a graph $\Gamma$ acting on the vertex set $\V\Gamma$,  we can endow $G$ with the {\em permutation topology}, see for instance \cite{Woess1991} and \cite{Moller2010}.   One description of this topology is as follows:  A neighbourhood basis of the identity element consists of all subgroups of the form $G_{(\Phi)}$, where $\Phi$ ranges over all finite subsets of $\Omega$.  Thus a subgroup is open if and only if it contains the pointwise stabilizer of some finite subset of $\Omega$.  Another way to describe this topology is to think of $\Omega$ as having the discrete topology and then the permutation topology is the compact-open topology on $G$.  The compact-open topology has the property that the action map $\Omega\times G\to \Omega; (\alpha, g)\mapsto \alpha g$ is continuous.   If the group $G$ already has a topology and the stabilizer $G_\alpha$ of a point $\alpha\in \Omega$ is open, then the permutation topology is a subset of the topology on $G$.  
If the action of $G$ is faithful then $G$ is totally disconnected.  

Various topological properties of the permutation topology have natural descriptions in terms of the groups action.  For instance, a sequence $\{g_i\}$ of elements in $G$ converges to an element $g\in G$ in the permutation topology if and only if for each element $\alpha\in \Omega$ there is a number $N_\alpha$ such that if $i\geq N_\alpha$ then $\alpha g_i=\alpha g$.  We say that $G$ is {\em closed in the permutation topology} if $G^\Omega$ is closed in $\Sym(\Omega)$ with respect to the permutation topology on $\Sym(\Omega)$.  Compactness and co-compactness have natural descriptions in the permutation topology.

\begin{lemma}\label{lem:compact-cocompact}
Let $G$ be a group acting transitively on a countable set $\Omega$. Endow  $G$ with the permutation topology.  Assume $G$ is closed in the permutation topology.  Assume also that all suborbits are finite.   Then:
\begin{enumerate}
    \item {\rm (\cite[Lemma~1]{Woess1991})}  The stabilizer of a point $\alpha\in \Omega$ is compact.
    \item {\rm (\cite[Lemma~2]{Woess1991})}  A subset $A$ in $G$ has compact closure if and only if all orbits of $A$ are finite.
    \item {\rm (\cite[Proposition~1]{Nebbia2000}, cf.~\cite[Lemma~7.5]{Moller2002})}  A subgroup $H$ in $G$ is co-compact, i.e.~$G/H$ is compact, if and only if $H$ has only finitely many orbits on $\Omega$.
\end{enumerate}
\end{lemma}

Suppose that $G$ already has a given topology and acts transitively on a set $\Omega$ and the stabilizers of points are compact open subgroups.   Then any compact subset of $G$ will also be compact in the permutation topology and thus have finite orbits on $\Omega$.  

From the above it  also follows that if  $G$ is a closed subgroup of the automorphism group of a locally finite graph $\Gamma$, then $G$ with the permutation topology is a totally disconnected, locally compact group.  The reader may also find it reassuring to know that if $G$ is a permutation group on a countable set $\Omega$, then the permutation topology is metrizable:  Enumerate
the points in
$\Omega$ as $\alpha_1, \alpha_2, \ldots$.  Take two elements $g, h\in G$.
Let $n$ be the smallest number such that $\alpha_n g\neq \alpha_n h$
or $\alpha_n g^{-1}\neq
\alpha_n h^{-1}$.  Set $d(g,h)=1/2^n$.   Then $d$ is a metric on $G$
that induces the permutation topology.

\subsection{Cayley--Abels graphs}

The central concept in this work is  the concept of a Cayley--Abels graph for a compactly generated, totally disconnected, locally compact group.  
The fundamental result about totally disconnected, locally compact groups is the theorem of van Dantzig \cite{Dantzig1936} that says that such a group contains a compact open subgroup. 

\begin{definition}
  Let $G$ be a totally disconnected, locally compact group. A connected, locally finite graph that $G$ acts vertex-transitively on such that the stabilizers of vertices are compact open subgroups is a called a \emph{Cayley--Abels graph} for $G$. 
\end{definition}

A Cayley--Abels graph for a totally disconnected, locally compact group exists if and only if the group is compactly generated.  It is easy to show that if $\Gamma$ is a Cayley--Abels graph and $\alpha$ a vertex and $g_1, \ldots, g_d$ are elements in $G$ such that $\{\alpha g_1, \ldots, \alpha g_d\}=\Gamma(\alpha)$, then the subgroup $\langle g_1, \ldots, g_d\rangle$ is transitive and the compact set $G_\alpha\cup\{g_1, \ldots, g_d\}$ generates $G$.  As a consequence of this we see that if $g_1,\ldots, g_k$ are elements in $G$ such that the set $\{\alpha g_1, \ldots, \alpha g_k\}$ contains a representative of every orbit of $G_\alpha$ on $\Gamma(\alpha)$ then the set $G_\alpha\cup\{g_1, \ldots, g_k\}$ generates $G$.  These facts will be used repeatedly. Two different ways of constructing a Cayley--Abels graph for a compactly generated, totally disconnected, locally compact group $G$ are described below.  

The first construction (see \cite[Construction 1]{KronMoller2008}) goes as follows.  Start with the set $G/U$ of right cosets of some compact open subgroup $U$ of $G$.  This will be the vertex set of our graph.  Then choose group elements $g_1, \ldots, g_n$ such that $G=\langle U, g_1, \ldots, g_n\rangle$ and finally take some element $\alpha\in G/U$ and let the edge set be the union of the $G$-orbits $\{\alpha, \alpha g_1\}G, \ldots, \{\alpha, \alpha g_n\}G$.  The resulting graph is a Cayley--Abels graph for $G$.

The second construction is in Abels' paper \cite[Beispiel 5.2]{Abels1974} (see also \cite[Construction 2]{KronMoller2008}).  In Abels' construction you start by taking a compact generating set $S$ and a compact open subgroup $U$.  Then construct the Cayley graph with respect to this generating set $S$ and finally you contract the right cosets of $U$.  After all this, you are left with a locally finite connected graph that $G$ acts transitively on.  The stabilizers of vertices are $U$ and its conjugates.  

In both construction the constructor has choices and thus there will always be more than one possible Cayley--Abels graph.  But they are all locally finite and thus the following concept is well defined.

\begin{definition}
Let $G$ be a compactly generated, totally disconnected, locally compact group.  The number  $\md(G)$ denotes the minimal degree of a Cayley--Abels graph for $G$. 
\end{definition}

\begin{remark}
Any two Cayley--Abels graphs for a compactly generated, totally disconnected group $G$ are quasi-isometric to each other (see \cite[Theorem 2.7]{KronMoller2008}). This implies that the two graphs have the same \lq\lq large scale\rq\rq\ properties, e.g.\ any two Cayley--Abels graph of a group $G$ have the same number of ends.
In this work we will not need the concept of quasi-isometry, except that it will appear again in a remark in Section~\ref{sec:discrete} and in Section~\ref{sec:degree2or3}.
\end{remark}

\subsection{The scale function and tidy supgroups}

The pioneering paper of Willis \cite{Willis1994} started a new wave of interest in totally disconnected, locally compact groups.  The fundamental concepts of Willis's theory are the {\em scale function} and {\em tidy subgroups}.  In a later paper, \cite{Willis2001}, Willis gave the following simple definitions of these fundamental concepts and showed that these new definitions are equivalent to his original definitions in \cite{Willis1994}.

\begin{definition} {\rm (See \cite[Theorem~3.1]{Willis2001})} \label{def:scale}
  Let \(G\) be a totally disconnected, locally compact group. The {\it scale function} on \(G\) is defined as
  \[s(g) =\min\{|U:U\cap g^{-1}Ug| \mid U \text{ compact open subgroup of } G\}.\]
  A compact open subgroup \(U\) of \(G\) is said to be {\it tidy for \(g\)} if and only if this minimum is attained at $U$, i.e. \(s(g) = |U:U\cap g^{-1}Ug|\).
\end{definition} 

The following proposition allows us to compute the scale function of elements by using an action of the group on a set.  

\begin{proposition}\label{prop:scale}{\rm (\cite[Corollary 7.8]{Moller2002})}
Let \(G\) be a totally disconnected, locally compact group and let $U$ be a compact, open subgroup of $G$.  Consider the action of $G$ on the set of right cosets $\Omega=G/U$.  Set $\alpha=U$ and think of $\alpha$ as a point in $\Omega$. If $g\in G$ then 
$$s(g)=\lim_{n\to \infty} |(\alpha g^n) G_\alpha|^{1/n},$$
and, furthermore, $s(g)=1$ if and only if there is a constant $C$ such that $|(\alpha g^i)G_\alpha|\leq C$ for all $i=0, 1, 2, \ldots$.
\end{proposition}

\section{Discrete actions}\label{sec:discrete}

Every discrete group is trivially a totally disconnected, locally compact group, but then the topology carries no information about the group. It is useful to know when the given group topology, or the permutation topology arising from the action of a group on a Cayley--Abels graph, carry additional information about the group and its action on a Cayley--Abels graph.

\begin{definition}
An action of a group $G$ on a set $\Omega$ is said to be {\em discrete} if  the stabilizers in $G^\Omega$ of points in $\Omega$ are finite.
\end{definition}

If the action of $G$ on $\Omega$ is discrete, then there is a finite set $\Phi\subseteq \Omega$ such that the pointwise stabilizer in $G^\Omega$ of $\Phi$ is trivial. In particular the permutation topology on $G^\Omega$ is discrete.  The pointwise stabilizer in $G$ of $\Phi$ is then a compact open subgroup of $G$, where $G$ has the permutation topology, and is equal to $K$, the kernel of the action.   
If the action is discrete and faithful, the permutation topology on $G$ is discrete.  

\begin{definition}
A topological group is said to be {\em nearly discrete} if it has a compact, open, normal subgroup.
\end{definition}

Following Cornulier in \cite{Cornulier2019} we call an action on a set $\Omega$ {\em block-discrete} if there is a $G$-congruence $\sim$ on $\Omega$ with finite classes such that stabilizers in $G^{\Omega/\sim}$ of points in $\Omega/\sim$ are finite.

\begin{lemma}{\rm (See \cite[Fact 5.6]{Cornulier2019})}
Let $G$ be a topological group acting transitively on a set $\Omega$ such that the stabilizers of points are compact, open subgroups.  Then $G$ is nearly discrete if and only if the action of $G$ on $\Omega$ is block-discrete.
\end{lemma}

\begin{proof}
Suppose that $G$ is nearly discrete and that $N$ is a compact, open, normal subgroup.  The orbits of $\omega N$ with $\omega \in \Omega$ are finite. They form the classes of a $G$-congruence on $\Omega$, denote the quotient space by $\Omega/N$.  Let $\omega \in \Omega$ and $U := G_{\omega N}$, where $\omega N \in \Omega/N$.  Then $U = G_{\omega N}$ with $\omega N \subseteq\Omega$. But $\omega N$ is finite, hence $U$ is a compact open subgoup of $G$.      The kernel $K$ of the action of $G$ on $\Omega/N$ contains $N$ and is contained in $G_{\omega N}$.  Thus $K$ is both open and compact. In particular $G/K$ is discrete and the image of $U$ in $G/K$ is finite.   Hence the action on $\Omega/N$ is nearly discrete.    

Assume now that the action of $G$ on $\Omega$ is block-discrete.  Hence there is a $G$-congruence $\sim$ on $\Omega$ with finite classes such that stabilizers in $G^{\Omega/\sim}$ are finite.   Let $K$ be the kernel of the action of $G$ on $\Omega/\sim$.  The action of $G$ on $\Omega/\sim$ is discrete and thus there is a finite set $\Phi \subseteq\Omega/\sim$ such that $G_{(\Phi)} = K$.  But $G_{(\Phi)}$ is a compact, open subgroup of $G$.  Thus $K$ is a compact, open, normal subgroup of $G$ and $G$ is nearly discrete. 
\end{proof}

The following nice result of Schlichting, here rephrased using our terminology, gives further connections between the action and the property that a group with the permutation topology is nearly discrete.

\begin{theorem}  {\rm (\cite{Schlichting1980})}  Let $G$ be a group acting transitively on a set $\Omega$. Endow $G$ with the permutation topology.   Then $G$ is nearly discrete if and only if there is a finite upper bound on the sizes of suborbits .
\end{theorem}

Turning to Cayley--Abels graphs we get the following corollary.

\begin{corollary}
Let $G$ be a compactly generated, totally disconnected, locally compact group.  Then $G$ is nearly discrete if and only if $G$ has a discrete action on some Cayley--Abels graph.
\end{corollary}

\begin{proof}
Suppose $G$ is nearly discrete and that $K$ is a compact, open, normal subgroup. Assume $\Gamma$ is a Cayley--Abels graph for $G$.   Then $\Gamma/K$ is connected and locally finite (recall that by part 3 in Lemma~\ref{lem:compact-cocompact} the orbits of $K$ are finite).  Hence $\Gamma/K$ is also a Cayley--Abels graph for $G$ and the action of $G$ on $\Gamma/K$ is discrete.

Conversely, if $G$ has a discrete action on some Cayley--Abels graph then we have already seen that $G$ is nearly discrete.  
\end{proof}

 \begin{remark}  Let $G$ be a compactly generated, totally disconnected, locally compact group and $\Gamma$ a Cayley--Abels graph for $G$. Let $H \leq G$ be a closed, co-compact subgroup.  Then $H$ itself is compactly generated and every Cayley--Abels graph for $H$ is quasi-isometric to $\Gamma$, see \cite[Corollary 2.11]{KronMoller2008}.  Suppose now that $K$ is a compact, normal subgroup of $G$.  Then $\Gamma/K$ is a Cayley--Abels graph for both $G$ and $G/K$.  Thus $\Gamma$ and $\Gamma/K$ are quasi-isometric. Conclusively, all Cayley--Abels graphs for $G$, $H$ and $G/K$ are quasi-isometric.
 \end{remark}

 If $G$ is nearly discrete then one can show by using a construction of Sabidussi, \cite[Theorem 4]{Sabidussi1964}, that there is a Cayley--Abels graph $\Gamma$ such that $G^\Gamma$ acts regularly on $\Gamma$. 

In some cases one can determine from the local action on a Cayley--Abels graph that the action on the graph is discrete.

\begin{lemma}\label{lem:discrete}
Let $G$ be a compactly generated, totally disconnected, locally compact group.  Suppose $\Gamma$ is a Cayley--Abels graph for $G$.  If the local action of $G$ is trivial then $G^\Gamma$ acts regularly on $\Gamma$ and if the local action of $G^\Gamma$ is free then $G^\Gamma$ acts freely on the arcs of $\Gamma$.
In particular, $G$ acts discretely on $\Gamma$.
\end{lemma}

\begin{proof}
We first consider the case where the local action is trivial.  If $\alpha\in\V\Gamma$, then $G_\alpha$ fixes every vertex in $\Gamma(\alpha)$.  If $\beta\in \Gamma(\alpha)$ then $G_\alpha$ fixes $\beta$ and thus every vertex in $\Gamma(\beta)$ is also fixed.  Since the graph $\Gamma$ is connected we see that $G_\alpha$ fixes every vertex in $\Gamma$, i.e.\ $G^\Gamma$ acts regularly.

Suppose the action of $G_\alpha$ on $\Gamma(\alpha)$ is free.  Let $\beta$ be a vertex in $\Gamma(\alpha)$.  By assumption $G_{\alpha\beta}$ fixes every vertex in $\Gamma(\alpha)$ and also every vertex in $\Gamma(\beta)$.   The conclusion now follows again from the connectivity of $\Gamma$. 
\end{proof}

Next we look at compact, normal subgroups in relation to minimal degree Cayley--Abels graphs.  The following lemma is well-known.

\begin{lemma}\label{lem:compact normal subgroups}
Let $\Gamma$ be a locally finite graph and $G \leq \Aut(\Gamma)$ a vertex-transitive subgroup.
 Suppose $N$ is a normal subgroup of $G$.
 Then, the degree of $\Gamma$ is bigger or equal to the degree of $\Gamma/N$. Equality holds if and only if no vertices of distance at most $2$ lie in the same $N$-orbit.
\end{lemma}

\begin{proof}
Consider two adjacent vertices $A$ and $B$ in $\Gamma/N$.    Since $A$ and $B$ are adjacent in $\Gamma/N$ there are vertices $\alpha, \beta\in \V\Gamma$ such that $\alpha\in A$, $\beta\in B$ and $\alpha$ and $\beta$ are adjacent in $\Gamma$.  Both $A$ and $B$ are $N$-orbits in $\V\Gamma$ and hence we see that every vertex in $A$  has a neighbour in $B$.   
Hence the degree of $\Gamma/N$ is at most equal to the degree of $\Gamma$.  

The degree of $\Gamma$ equals the degree of $\Gamma/N$ if and only if for every vertex $\alpha\in \V\Gamma$ its neighbours belong to different $N$-orbits and none of them belongs to the $N$-orbit of $\alpha$. Thus the stabilizer of $\alpha$ in $N$ must act trivially on the neighbourhood of $\alpha$ in $\Gamma$.    The same argument as in the proof of Lemma~\ref{lem:discrete} shows now that the stabilizer in $N$ of a vertex in $\Gamma$ acts trivially on $\Gamma$.
\end{proof}

This lemma directly implies the following.

\begin{corollary}
Let $G$ be a compactly generated, totally disconnected, locally compact group.  Suppose $\Gamma$ is a Cayley--Abels graph for $G$ with minimal degree.
If $N$ is a compact, normal subgroup of $G$, then $N^\Gamma$ acts freely on $\V\Gamma$ and no two vertices in the same $N$-orbit are adjacent.
In particular $N^\Gamma$ is finite.
\end{corollary}

\begin{corollary}\label{cor:nearly discrete implies discrete action on a minimal CA graph}
Let $G$ be  a compactly generated, totally disconnected, locally compact group.  If $G$ is nearly discrete, then $G$ acts discretely on some minimal degree Cayley--Abels graph.
\end{corollary}

\begin{proof}
Let $\Gamma$ be some minimal degree Cayley--Abels graph.  If $K$ is a compact, open, normal subgroup then $\Gamma/K$ is also a minmal degree Cayley--Abels graph for $G$ and $G$ acts discretely on $\Gamma/K$.
\end{proof}

\section{Groups with minimal degree 2 or 3}\label{sec:degree2or3}

In this section we first characterize those compactly generated, totally disconnected, locally compact groups where the minimal degree of a Cayley--Abels graph is $2$. Note that a $2$-regular minimal degree Cayley--Abels graph can only be the infinite line.  We also look at the case when the minimal degree is equal to 3, for proofs of those results the reader is referred to a companion paper, \cite{ArnadottirLederleMoller2020}, to this paper.

Let $\Gamma$ be any graph. A \emph{ray} is a $1$-way infinite arc such that all its vertices are different.
An \emph{end} of $\Gamma$ is an equivalence class of rays, where the ray $(\alpha_0,\alpha_1,\dots)$ is equivalent to $(\beta_0,\beta_1,\dots)$ if and only if there is a ray $(\gamma_0,\gamma_1,\dots)$ such that the intersections $\{\gamma_i \mid i \geq 0\} \cap \{\alpha_i \mid i \geq 0\}$ and $\{\gamma_i \mid i \geq 0\} \cap \{\beta_i \mid i \geq 0\}$ are both infinite. 

If two graphs are quasi-isometric then they have the same number of ends, see 
\cite[Proposition~ 1]{Moller1992a}.  As any two Cayley--Abels graphs of a group $G$ are quasi-isometric (see \cite[Theorem 2.7]{KronMoller2008}), we can define the number of ends of a compactly generated, totally disconnected, locally compact group as the number of ends of a Cayley--Abels graph.  

\begin{theorem}\label{T2ends}
 For a compactly generated, totally disconnected, locally compact group $G$ the following conditions are equivalent.
 \begin{enumerate}
     \item The minimal degree of a Cayley--Abels graph for $G$ is 2.
     \item The group $G$ has precisely two ends.
     \item There is a continuous surjective homomorphism with a compact, open kernel from $G$ onto the infinite cyclic group or the infinite dihedral group (both with the discrete topology).
     \item The group $G$ has a co-compact cyclic discrete subgroup.
 \end{enumerate}
\end{theorem}

\begin{proof}
First note that (1) implies (2) trivially because if $\md(G)=2$ then the integer graph $\ZZ$ is a Cayley--Abels graph for $G$, so $G$ has two ends.  Then (2) implies (3) by \cite[Satz 4.5]{Abels1974}, see also \cite[Proposition 3.2]{MollerSeifter1998}. 
To prove that (3) implies (1), note that
the groups $\ZZ$ and $D_\infty$ both have regular actions on the integer graph $\ZZ$. From this we see that $G$ has an action on the integer graph $\ZZ$ with a compact, open kernel and thus $\ZZ$ is an Cayley--Abels graph for $G$.

Finally we note that (2) and (4) are equivalent.  This can be seen, for example, from results in Section 5 in the paper by Jung and Watkins, \cite{JungWatkins1984} or in Abels' paper \cite[Satz 3.10]{Abels1974}.  A self-contained proof of this fact can be found in \cite[Appendix C]{ArnadottirLederleMoller2020}.
\end{proof}

\begin{remark}
In the case of a finitely generated group acting on a locally finite Cayley graph it follows from results of Hopf \cite[Satz 5]{Hopf1944} and Wall \cite[Lemma 4.1]{Wall1967} that conditions (2), (3) and (4) in the theorem above are equivalent.
\end{remark}

In \cite{ArnadottirLederleMoller2020} we study compactly generated, totally disconnected, locally compact groups that have minimal degree 3. The first part of the following theorem was previously known by experts.

\begin{theorem}{\rm{(\cite[Sections 3 and 4]{ArnadottirLederleMoller2020})}}
Let $G$ be a compactly generated, totally disconnected, locally compact group that is not nearly discrete.    Assume that $\md(G)=3$ and let $\Gamma$ be a minimal degree Cayley--Abels graph for $G$.  Then, one of following is true.
\begin{enumerate}
    \item The action of $G$ is transitive on the edges and $\Gamma$ is a tree.
    \item The action of $G$ has precisely two orbits on the edges of $\Gamma$. For every $s \geq 1$, the group $G$ acts with precisely two orbits on the subset of $s$-arcs whose underlying edges lie alternately in the two edge orbits of $G$.
\end{enumerate}
\end{theorem}

A totally disconnected, locally compact group $G$ is said to be {\em uniscalar} if the scale function is constant, i.e.\ $s(g)=1$ for all $g\in G$.  Equivalently, every element in the group normalizes some compact, open subgroup.

\begin{theorem}{\rm{(\cite[Theorem 6.2]{ArnadottirLederleMoller2020})}}\label{thm:not-uniscalar}
Suppose $G$ is a compactly generated, totally disconnected, locally compact group that is not nearly discrete.  If $\md(G)=3$ then $G$ is not uniscalar.   
\end{theorem}
  
If a totally disconnected, locally compact group $G$ has a compact open normal subgroup then $G$ is uniscalar.  Bhattacharjee and Macpherson \cite[Section 3]{BhattacharjeeMacpherson2003} (following up on work by Kepert and Willis, \cite{KepertWillis2001}), constructed an example of a compactly generated, totally disconnected, locally compact group that has no compact, open, normal subgroup, but every element normalizes some compact open subgroup.  

\begin{corollary}{\rm{(\cite[Corollary 6.3]{ArnadottirLederleMoller2020})}}\label{cor:uniscalar implies nearly discrete}
Let $G$ be a compactly generated, totally disconnected, locally compact group having a 3-regular Cayley--Abels graph.
If every $g \in G$ normalizes a compact open subgroup of $G$ then $G$ has a compact, open, normal subgroup.
\end{corollary}

\section{Cayley--Abels graphs and the modular function}

Let $G$ be a locally compact group and $\mu$ a right-invariant Haar-measure on $G$. The \emph{modular function} is the map $\Delta \colon G \to \mathbb{R}^+$ such that for every measurable subset $A \subseteq G$ and for every $g \in G$ we have $\mu(gA)=\Delta(g)\mu(A)$; it is well-known that it exists and is a homomorphism.   A group is {\em unimodular} if $\Delta(g)=1$ for all $g\in G$.  

\subsection{Reading the modular function off the edges}\label{sec:modular}

The following result linking the modular function and sizes of suborbits was proved by Schlichting \cite[Lemma~1]{Schlichting1979}, see also 
\cite[Theorem~1]{Trofimov1985}.   

\begin{lemma}\label{LSuborbits}
Let $G$ be a totally disconnected, locally compact group and $\Delta$ the modular function on $G$. If $U$ is a compact open subgroup of $G$ and $g\in G$ then
$$\Delta(g)=\frac{|U:U\cap g^{-1}Ug|}{|g^{-1}Ug:U\cap g^{-1}Ug|}.$$
Suppose $G$ acts on a set $\Omega$ with compact, open point stabilizers.  If $\alpha\in\Omega$ and $g\in G$ then 
\begin{equation}\label{eq:modular}\Delta(g)=\frac{|G_\alpha:G_\alpha\cap g^{-1}G_\alpha g|}{|g^{-1}G_\alpha g:G_\alpha\cap g^{-1}G_\alpha g|}
=\frac{|(\alpha g) G_\alpha|}{|\alpha G_{ \alpha g}|}.
\end{equation}

Furthermore, the image of the modular function is a subgroup of the multiplicative group of positive rational numbers. 
\end{lemma}

If $G$ acts on a set $\Omega$ with compact point stabilizers and if for $g, h\in G$ there is an $\alpha\in \Omega$ with $\alpha g=\alpha h$, then $\Delta(g)=\Delta(h)$.   If in addition $G$ is unimodular and point stabilizers are open, then $|(\alpha g) G_\alpha|=|\alpha G_{\alpha g}|$ for all $g$ in $G$ and all $\alpha\in\Omega$.

The following idea originates from the paper by Bass and Kulkarni \cite[Section 3]{BassKulkarni1990} and enables us to \lq\lq read\rq\rq\ the values of the modular function for a compactly generated, totally disconnected, locally compact group $G$ off a Cayley--Abels graph $\Gamma$.      Label an arc $(\alpha,\beta)$ in $\Gamma$
with the number
$$\Delta_{(\alpha, \beta)}=\frac{|\beta G_\alpha|}{|\alpha G_\beta
|}.$$
Note that $\Delta_{(\beta, \alpha)}=\Delta_{(\alpha, \beta)}^{-1}$.
This arc-labelling is clearly invariant under the action of $G$ on the set of arcs.  Combining these observations, we see that if $\Delta_{(\alpha, \beta)}\neq 1$ there can not exists $g\in G$ such that $(\alpha,\beta) g=(\beta,\alpha)$. 
If $g \in G$ satisfies $\alpha g=\beta$ then Lemma \ref{LSuborbits} readily implies that
\begin{equation} \label{Eq:Deltag}
   \Delta(g)= \frac{|\beta G_\alpha|}{|\alpha G_\beta|}=\Delta_{(\alpha, \beta)}.
\end{equation}
Suppose now that $\gamma$ is a neighbour of $\beta$ and $h \in G$ satisfies  $\alpha h=\gamma$. 
Set $g'= g^{-1}h \in G$.  Then $\beta g'=(\alpha g)g^{-1}h=\gamma$.
Hence $\alpha h=\alpha g g'$ and we see that
$$\Delta(h)=\Delta(gg')=\Delta(g)\Delta(g')=\Delta_{(\alpha,\beta)}\Delta_{(\beta,\gamma)}.$$
Inductively, we get that if $\alpha,\beta \in \Gamma$ are arbitrary vertices, $g \in G$ satisfies $\alpha g=\beta$ and $(\alpha_0,\dots,\alpha_s)$ is an $s$-arc with $\alpha_0=\alpha$ and $\alpha_s=\beta$, then
$$\Delta(g)=\Delta_{(\alpha_0,\alpha_1)}\cdot\ldots\cdot \Delta_{(\alpha_{s-1},\alpha_s)}.$$
Thus the labelled Cayley--Abels graph completely describes the modular function on
$G$. 
We have now proved the following:

\begin{theorem}\label{TModulargraph}
Let $G$ be a compactly generated, totally disconnected, locally compact group, $\Delta$ the modular function on $G$ and $\Gamma$ a Cayley--Abels graph for $G$.
If $g \in G$ and $(\alpha_0,\dots,\alpha_s)$ is a $s$-arc in $\Gamma$ such that $\alpha_0g = \alpha_s$,
then $\Delta(g) = \Delta_{(\alpha_0,\alpha_1)} \dots \Delta_{(\alpha_{s-1},\alpha_s)}$.  In particular, the image of $\Delta$ is generated by the labels of the arcs.
\end{theorem}

\begin{corollary} \label{cor:generators of ImDelta}
Let $G$ be a \cgtdlc group, $\Delta$ the modular function on $G$ and $\Gamma$ a Cayley--Abels graph for $G$.
For a vertex  $\alpha$ in $\Gamma$,
let $B_1,\dots,B_n$ be the orbits of $G_\alpha$ on the neighbours of $\alpha$ and choose, for each $i$, an element $g_i \in G$ with $\alpha g_i \in B_i$.
Then
\[
\im(\Delta) = \left\langle \frac{|(\alpha g_1) G_\alpha|}{|(\alpha g_1^{-1}) G_\alpha|}, \dots, \frac{|(\alpha g_n) G_\alpha|}{|(\alpha g_n^{-1}) G_\alpha|} \right\rangle \leq \mathbb{Q}^+.
\]
The image of $\Delta$ is a finitely generated, free abelian subgroup of the multiplicative subgroup of the positive rational numbers.
\end{corollary}

\begin{proof}
Recall that $G = \langle G_\alpha, g_1, \dots g_n \rangle$.  Lemma~\ref{LSuborbits} implies that $G_\alpha$ is contained in the kernel of the modular function.
Thus $\im(\Delta)$ is generated by $\Delta(g_1),\dots,\Delta(g_n)$.
Note that
$|\alpha G_{\alpha g_i}|= |\alpha (g_i^{-1} G_\alpha g_i)| =|(\alpha g_i^{-1}) G_\alpha|$.
So (\ref{Eq:Deltag}) gives us $\Delta(g_i) = \frac{|\alpha g_i G_\alpha|}{|\alpha g_i^{-1} G_\alpha|}$.   The last statement follows since finitely generated subgroups of the multiplicative subgroup of positive rational numbers are free abelian.
\end{proof}

\subsection{Minimal degree and the modular function}

The observations in the last section relating the modular function and Cayley--Abels graphs allow us to get a lower bound on the minimal degree of a Cayley--Abels graph in terms of the values of the modular function.

\begin{theorem}\label{C-Adigraphs}
  Let $G$ be a non-compact, compactly generated, totally disconnected, locally compact group. Let $H \leq \QQ^+$ be the image of the modular function on $G$.
  Then every Cayley--Abels graph of $G$ has degree at least \(\min(A)\), where 
  \[A = \Bigg\{\sum_{i=1}^k(p_i+q_i) \ \Bigg| \ p_1,\dots,p_k,q_1,\dots,q_k \in \NN^*, \,
  \left \langle \frac{p_1}{q_1}, \dots, \frac{p_k}{q_k} \right \rangle=H\Bigg\}\subseteq \NN^*.\]
In particular, if $H$ is cyclic and generated by the rational number \(p/q\), where \(p,q\in \NN^*\) are relatively prime, then every Cayley--Abels graph of \(G\) has degree at least \(p+q\).
\end{theorem}

\begin{proof} 
Let $\Gamma$ be a Cayley--Abels graph for $G$ and let $\alpha \in \V\Gamma$.
 Think of $\Gamma$ as a labelled digraph as in Section \ref{sec:modular}. The group \(G\)  has finitely many orbits \(E_1,\dots, E_n\) on its edges. By vertex-transitivity, we can take an edge $e_i=\{\alpha, \beta_i\}\in E_i$ and set $p_i=|\beta_i G_{\alpha}|$ and $q_i=|\alpha G_{\beta_i}|$. The label of the arc $(\alpha, \beta_i)$ is $p_i/q_i$.  These numbers are independent of the choice of the representative $e_i$ of $E_i$.
 Note that if $g_i\in G$ is such that $\alpha g_i=\beta_i$ then $|\alpha G_{\beta_i}|=|(\alpha g_i^{-1})G_{\alpha}|$.  If $p_i\neq q_i$ then the sets $\beta_i G_{\alpha}$ and $(\alpha g_i^{-1}) G_{\alpha}$ are disjoint and are both contained in $\Gamma(\alpha)$.  Thus the contribution of $E_i$ to the degree of $\alpha$ is at least $p_i+q_i$. 
 Hence the degree of $\Gamma$ is at least equal to the sum of all $p_i+q_i$ with $p_i\neq q_i$.  Corollary \ref{cor:generators of ImDelta} now says 
  \[H=\left\langle \frac{p_1}{q_1},\dots, \frac{p_n}{q_n}\right\rangle\]
and the result follows.
\end{proof}

The lower bound provided in Theorem \ref{C-Adigraphs} is far away from  being sharp in general:  For a unimodular group the bound provided is $2$ but the minimal degree of a unimodular group can be arbitrarily large.

Theorem \ref{C-Adigraphs} does not say anything meaningful about the minimal degrees of unimodular groups, such as the automorphism groups of regular or bi-regular trees.  (A tree is said to be {\em bi-regular}, or more precisely $(d,d')$-{\em bi-regular}, if all the vertices in one part of the natural bipartition have degree $d$ and all the vertices in the other part have degree $d'$.)  But Theorem \ref{C-Adigraphs} can be applied to the stabilizer of an end in a regular or bi-regular tree.
Ends of graphs were defined in Section~\ref{sec:degree2or3}, but in the special case of trees the definition has a simpler form:
Two rays $(\alpha_0,\alpha_1,\dots)$ and $(\beta_0,\beta_1,\dots)$ in a tree $T$ are said to be equivalent if  and only if they have the same infinite tail, i.e. there exist $k,l \geq 0$ such that $(\alpha_k,\alpha_{k+1},\alpha_{k+2},\dots)=(\beta_l,\beta_{l+1},\beta_{l+2},\dots)$. An \emph{end} of a tree $T$ is an equivalence class of a rays.   Given an end $\omega$ of a tree, it is easy to show that for each vertex $\alpha$ in $T$ there is a unique ray $R_\alpha$ in $\omega$ that has $\alpha$ as its initial vertex.
It is obvious that $\Aut(T)$ acts on the set of ends of $T$.  If $\omega$ is an end of $T$, then the stabilizer of the end $\omega$, the group $\aut(T)_\omega$, is a closed subgroup of $\aut(T)$.  Hence if $T$ is locally finite, the stabilizer of an end is a totally disconnected, locally compact group.

\begin{corollary}\label{cor:end stabilizer in tree}\(\)
\begin{enumerate}
    \item Let $T_d$ denote the $d$-regular tree and let  $\omega$ be an end of $T_d$.  Set $G=\aut(T)_\omega$.  Then  $\md(G)=d$.
\item Let $T_{d,d'}$ be a bi-regular tree for some distinct integers $d, d'\geq 2$.   Suppose $\omega$ is an end of $T$ and set $G=\aut(T)_\omega$.   Then $\md(G_\omega)= (d-1)(d'-1)+1$. 
    \item Let $\aut^+(T_d)$ denote the subgroup of $\aut(T_d)$ that leaves each part of the natural bipartition of $\V T_d$ invariant.  Suppose $\omega$ is an end of $T_d$ and set $G=\aut^+(T)_\omega$.  Then $m(G)=(d-1)^2+1$.
\end{enumerate}
\end{corollary}

\begin{proof} 
Clearly (3) is a special case of (2) with $d=d'$ and (1) follows from (2) by setting $d'=2$. It is left to prove (2).

Denote with $V_1$ one of the classes of the natural bipartition of $T_{d,d'}$.  Construct a new graph $T$ with vertex set $V_1$ and edge set 
$\{\{\alpha,\beta\}\mid d_{T_d}(\alpha, \beta)=2\mbox{ and } \beta\in R_\alpha\}$.
It is easy to see that this new graph is isomorphic to the $((d-1)(d'-1)+1)$-regular tree.  The end $\omega$ of $T_{d,d'}$ corresponds to an end of $T$ and $G$ acts on $T$ fixing that end. Clearly $G$ acts vertex- and edge-transitively on $T$.  So, $T$ is a Cayley--Abels graph for $G$ and $\md(G)\leq (d-1)(d'-1)+1$.  

Let $\{\alpha, \beta\}$ be an edge in $T$ such that $\beta\in R_\alpha$.  Suppose that $g\in G$ and $\alpha g=\beta$.  Then
$$\Delta(g)=\frac{|\beta G_\alpha|}{|\alpha G_\beta|}=\frac{1}{(d-1)(d'-1)}$$
and by edge-transitivity this rational number generates the image of $\Delta$.  Proposition \ref{cor:generators of ImDelta} and Theorem \ref{C-Adigraphs}  say that  $\md(G)\geq (d-1)(d'-1)+1$. Combining this with the upper bound above we see that $\md(G)=(d-1)(d'-1)+1$.
\end{proof}

\begin{example}\label{expl:tree}
Set $H=\aut^+(T_d)$.  
The authors don't know the exact value of $\md(H)$.  But $H$ acts transitively on the edge set of $T_d$ and thus the line graph of $T_d$ is a Cayley--Abels graph for $H$.  (The {\em line graph} of a graph $\Gamma$ has the set of edges in $\Gamma$ as a vertex set and two vertices in the line graph, i.e.\ edges in $\Gamma$, are adjacent in the line graph if and only if they have a common end vertex.) The degree of the line graph of $T_d$ is $2d-2$ and hence $\md(H)\leq 2d-2$.  When $d\geq 3$ we see that $\md(H)$ grows at most linearly with $d$ and is strictly smaller than the minimal degree of the subgroup fixing an end, which grows quadratically with $d$.

Similarly, the line graph of $T_{d, d'}$ is a Cayley--Abels graph for $H=\aut(T_{d, d'})$.  The degree of the line graph is $d+d'-2$ and hence $\md(H)\leq d+d'-2$.  If  $d, d'\geq 3$ then $\md(H)$ is strictly smaller than the minimal degree of the subgroup fixing an end. 
\end{example}

\subsection{Digraphs and the modular function}

One of the first applications of the modular function to the study of graphs and their automorphism groups was by  Praeger in 1991, \cite{Praeger1991}, where she proves the following theorem using the modular function.   (She does not use that term {\em modular function} explicitly but defines a function in terms of sizes of suborbits as in Formula (\ref{eq:modular}) in Lemma~\ref{LSuborbits}.)

Let $\vZZ$ denote the digraph with vertex set $\ZZ$ and arc set $\{(i, i+1)\mid i\in \ZZ\}$. Similarly, we define $\vZZ^n$ as the digraph with vertex set $\ZZ^n$, in which $((x_1, \ldots, x_n), (y_1, \ldots, y_n))$ is an arc if and only if there is $1 
\leq j \leq n$ such that $y_j=x_j+1$ and $x_i=y_i$ for all $i\neq j$.  
A digraph is said to have \emph{Property Z} if a surjective digraph 
morphism $\Gamma\to \vZZ$ exists.

A \emph{fibre} of a map is the pre-image of a point.

\begin{theorem}  {\rm (\cite{Praeger1991})}\label{TPraeger}
Let $\Gamma$ be an infinite, connected, vertex- and arc-transitive digraph with finite but unequal in- and out-degrees.  Then there exists a surjective graph morphism $\varphi:\Gamma\rightarrow \vZZ$, i.e.~$\Gamma$ has Property Z.  The fibres of $\varphi$ are all infinite.
\end{theorem}

Note that, since $\vZZ$ is vertex-transitive with in- and out-degree $1$, a digraph morphism from a connected digraph to $\vZZ$ is uniquely determined by the image of one vertex and this image can be chosen arbitrarily.

Praeger's theorem can be generalized to digraphs with more than one orbit on arcs.   The proof resembles the proof of Praeger's theorem found in Evans' paper \cite[Theorem~3.2]{Evans1997}.

\begin{theorem}\label{TPraeger-n}
Let $\Gamma$ be a locally finite digraph.  Let $G\leq \autga$ be a closed subgroup acting vertex-transitively on $\Gamma$. Assume $G$ is not unimodular and denote by $\Delta$ the modular function on $G$.  Suppose that $G$ has $n$ orbits on the arcs of $\Gamma$ and that $\im(\Delta)$ is a free abelian group of rank $n$.  Then, there exists a surjective digraph morphism $\varphi: \Gamma\to \vZZ^n$ with infinite fibres.
\end{theorem}

\begin{proof}  
Fix a vertex $\alpha$ in $\Gamma$.  By vertex-transitivity we can choose representatives $(\alpha, \beta_1), \ldots, (\alpha, \beta_n)$ for the $n$ orbits on the arcs of $\Gamma$. Then we find elements $g_i\in G$ such that $\alpha g_i=\beta_i$. Since $G=\langle G_\alpha,g_1, \dots,g_n \rangle$ we see that $\langle \Delta(g_1),\dots,\Delta(g_n) \rangle = \im(\Delta)\cong \ZZ^n$ and the Cayley digraph $\Theta$ of $\im(\Delta)$ with respect to this generating set is isomorphic to $\vZZ^n$.  

For a vertex $\beta$ in $\Gamma$ find $g\in G$ such that $\alpha g=\beta$ and set $\varphi(\beta)=\Delta(g)$.
By Lemma~\ref{LSuborbits} the value of $\varphi(\beta)$ is independent of the choice of $g$, so the  map $\varphi\colon \V\Gamma\rightarrow\V\Theta$ is well-defined.  

Let $(\beta, \gamma)$ be an arc in $\Gamma$.  Find $g\in G$ and $1 \leq i \leq n$ such that $(\alpha, \beta_i)g =(\beta, \gamma)$.  Then $\alpha g_ig=\gamma$.  Thus $\varphi(\beta)=\Delta(g)$ and $\varphi(\gamma)=\Delta(g_ig)=\Delta(g_i)\Delta(g)=\Delta(g)\Delta(g_i)$.  The image under $\varphi$ of the arc $(\beta, \gamma)$ is the pair $(\Delta(g), \Delta(g)\Delta(g_i))$ that is indeed an arc in $\Theta$. Hence $\varphi$ is a morphism of digraphs.

Consider the digraph with vertex set $\V\Gamma$ and arc set $(\alpha, \beta_1)G$.  Set $H=\langle G_\alpha, g_1\rangle$. Let $\Gamma'$ be the component of this digraph that contains $·\alpha$.  The group $H$ acts arc-transitively on $\Gamma'$ and $\varphi$ restricted to $\Gamma'$ gives a graph morphism whose image is isomorphic to $\vZZ$.  

By Theorem~\ref{TPraeger} and the explanation following it the fibers of $\varphi$ are infinite.
\end{proof}

\begin{corollary}   Let $\Gamma$ be a digraph satisfying the conditions in Theorem~\ref{TPraeger-n}.  Then:
\begin{enumerate}
\item there exists a surjective digraph morphism $\varphi:\Gamma\rightarrow \vZZ$, that is to say, $\Gamma$ has Property Z;
    \item $\Gamma$ is bipartite.
\end{enumerate}
\end{corollary}

\begin{proof}
The map $\vZZ^n\rightarrow \vZZ; (x_1, \ldots, x_n)\mapsto x_1+\cdots+x_n$ is clearly a surjective digraph morphism.  The composition of this map with the digraph morphism $\varphi$ from Theorem~\ref{TPraeger-n} we get a surjective digraph morphism $\Gamma\to \vZZ$.  

By the first part the digraph $\vZZ$ is the image of $\Gamma$ under a surjective morphism.  Since $\vZZ$ is bipartite the digraph $\Gamma$ must also be bipartite.
\end{proof}

\begin{example}\label{ex:Praeger-n}\(\)
\begin{enumerate}
    \item Let $p,q$ be positive, co-prime integers.
    Let $T$ be the regular directed tree with in-degree $2$ and out-degree $p+q$. Colour its arcs red and blue such that each vertex has $p$ red outgoing arcs, $q$ blue outgoing arcs and one incoming arc of each colour.
    Clearly, $\aut(T)$ is vertex- and arc-transitive. Let $G\leq\aut(T)$ be the subgroup that has the two colour classes as its arc orbits and let $\Delta$ denote the modular function on $G$.   By Theorem \ref{TModulargraph}, $\im(\Delta)$ is generated by $\{p,q\}$, and is therefore a free abelian group of rank two. Thus by Theorem \ref{TPraeger-n}, there is a surjective digraph morphism from $T$ to $\vZZ^2$ with infinite fibres.  Theorem~\ref{C-Adigraphs} implies that underlying undirected graph of $T$, the $(p+q+2)$-regular tree, is a minimal degree Cayley--Abels graph for $G$.
    \item Let $T_1,\dots,T_n$ be regular directed trees with in-degree one and out-degrees $d_1,\dots, d_n$ such that the $d_i$ are distinct prime numbers. Let $\Gamma=T_1\square\cdots\square T_n$ be the Cartesian product, that is, the vertex set is $\V T_1 \times \dots \times \V T_n$ and $((\alpha_1,\dots,\alpha_n),(\beta_1,\dots,\beta_n))$ is an arc if and only if there is $i$ such that $(\alpha_i,\beta_i)$ is an arc in $T_i$ and $\alpha_j = \beta_j$ for all $j \neq i$.
    In other words, it is the $1$-skeleton of the cube complex $T_1 \times \cdots \times T_n$, remembering the directions on the factors.
    Let $G=\autga$. It can be shown that $G = \aut(T_1)\times\cdots\times\aut(T_n)$; this is a consequence of the fact that two adjacent arcs are ``parallel to the same tree $T_i$'' if and only if they are not sides of the same (undirected) square (see also \cite[Corollary 6.12]{HammackImrichKlavzar2011} for a more general version). Now it is easy to see that $G$ has $n$ orbits on the arcs of $\Gamma$. Further, if $\Delta$ denotes the modular function on $G$, then $\im(\Delta)=\langle d_1,\dots,d_n\rangle$ is a free abelian group of rank $n$ (since the $d_i$ are distinct primes) and so there exists a surjective digraph morphism, $\varphi:\Gamma\to\vZZ^n$ with infinite fibres.
    Again Theorem \ref{C-Adigraphs} proves that $\Gamma$ is a minimal degree Cayley--Abels graph for $G$.
\end{enumerate}
\end{example}


The following proposition can be seen as an addendum to Theorem~\ref{TPraeger}.  

\begin{proposition}\label{Prop:coprime}
Let $\Gamma$ be a locally finite digraph and $G\leq \Aut(\Gamma)$ a subgroup acting arc- and vertex-transitively on $\Gamma$.   If the in- and out-degrees of $\Gamma$ are coprime then $G$ is highly-arc-transitive.  Furthermore, the subdigraph induced by the set of descendants of any vertex is a tree.
\end{proposition}
  
\begin{proof}
Let $q$ denote the in-degree of $\Gamma$ and let $p$ denote the out-degree. Take vertices  $\alpha$ and $\beta$ in $\Gamma$ such that $(\alpha, \beta)$ is an arc.  Choose some element $g\in G$ such that $\beta=\alpha g$.  The second formula for the modular function in Lemma~\ref{LSuborbits} now says that
$$\Delta(g)=\frac{|\beta G_\alpha|}{|\alpha G_\beta|} = \frac{p}{q}.$$
Then
$$\frac{|(\alpha g^n)G_\alpha|}{|\alpha G_{\alpha g^n}|}=\Delta(g^n)=\Delta(g)^n=\left(\frac{p}{q}\right)^n=\frac{p^n}{q^n}.$$
Note that $(\alpha, \alpha g, \ldots, \alpha g^n)$ is an $n$-arc in $\Gamma$.  There are $p^n$ distinct $n$-arcs in $\Gamma$ that start at $\alpha$ and each vertex in the orbit $(\alpha g^n)G_\alpha$ is the terminal vertex of such an arc.   Thus $|(\alpha g^n)G_\alpha|\leq p^n$.  But since $|(\alpha g^n)G_\alpha|/|\alpha G_{\alpha g^n}|=p^n/q^n$ and $p$ and $q$ are coprime we see that $|(\alpha g^n)G_\alpha|=p^n$ and $|\alpha G_{\alpha g^n}|=q^n$.   

If $\gamma$ is the terminal vertex of some $n$-arc starting at $\alpha$ then the orbit of $\gamma$ under $G_\alpha$ has precisely $p^n$ elements.  The number of $n$-arcs starting at $\alpha$ is $p^n$ and we see that $G_\alpha$ acts transitively on the set of $n$-arcs starting at $\alpha$.  As $G$ acts transitively on $\V \Gamma$ we see that $G$ acts transitively on the $n$-arcs in $\Gamma$.  Hence $\Gamma$ is highly-arc-transitive.  

Every arc in the subgraph induced by the set of descendants of $\alpha$ is contained in some $n$-arc staring at $\alpha$.  
No two distinct $n$-arcs starting at $\alpha$ have a common terminal vertex.  Thus the subgraph induced by the set of descendants of $\alpha$ is tree.
\end{proof}

In \cite{Moller2002} the scale function and tidy subgroups are analysed by using graph theoretical concepts.  A prominent role in this analysis is played by highly-arc-transitive digraphs such that the subgraph induced by the set of descendants of a vertex is a tree.  It is therefore not surprising that the above proposition can be interpreted as a result about tidy subgroups.

\begin{corollary}
  Let \(G\) be a totally disconnected, locally compact group.  Suppose $g\in G$ and $U$ is a compact, open subgroup of $G$ such that the two indices $|U:U\cap g^{-1}Ug|$ and $|U:U\cap gUg^{-1}|$ are coprime.  Then $U$ is tidy for $g$ and $s(g)=|U:U\cap g^{-1}Ug|$.
\end{corollary}

\begin{proof}
We define a digraph $\Gamma$ such that the vertex set is the set of cosets $G/U$ and the set of arcs is the $G$-orbit $(\alpha, \beta)G$ where $\alpha=U$ and $\beta=Ug$.  Note that $G_\alpha=U$.  Then $\Gamma$ is an arc- and vertex-transitive digraph.  The out-degree is equal to $p=|U:U\cap g^{-1}Ug|$ and the in-degree is equal to $q=|U:U\cap gUg^{-1}|$.  Since the in- and out-degrees are coprime the last proposition applies and therefore, the digraph $\Gamma$ is highly-arc-transitive and the subgraph induced by the set of descendants of a vertex is a tree.  We also see that 
$$|U:U\cap g^{-n}Ug^n|=|(\alpha g^n)G_\alpha|=p^n=|U:U\cap g^{-1}Ug|^n.$$
Now \cite[Corollary 3.5]{Moller2002} says precisely that $U$ is tidy for $g$. 
\end{proof}

The following proposition is an $n$-dimensional version of Proposition~\ref{Prop:coprime} and can be proved in a similar way.  

\begin{proposition}
Let $\Gamma$ be a connected, locally finite digraph.  Let $G\leq \Aut(\Gamma)$ be a vertex-transitive subgroup with $n$ orbits $A_1,\dots,A_n$ on the arcs of $\Gamma$ and denote by $d_1^-, \ldots, d_n^-, d_1^+, \ldots, d_n^+$ the respective in- and out-degrees of these orbits at any given vertex.
Let $f \colon \A\Gamma \to \{1,\dots,n\}$ be the unique map satisfying $f(A_i)=i$ and denote with $f^s$ the induced map from the set of $s$-arcs of $\Gamma$ to $\{1,\dots,n\}^s$.

If each of the numbers $d_1^-, \ldots, d_n^-$ is coprime with each of the numbers $d_1^+, \ldots, d_n^+$, then $G$ acts transitively on each fibre of $f^s$.
\end{proposition}

\section{Simple composition factors and the scale function}\label{sec:composition}

In this section we study the interplay between the structure of compact, open subgroups, the scale function and the minimal degree of a Cayley--Abels graph.  In order to do so we must first proof a version of the Jordan--H\"older theorem for second countable, profinie groups and study the composition factors of compact open subgroups.  
For the reader not familiar with those terms it is enough to know that every compact subgroup of the automorphism group of a locally finite, connected graph is second countable and profinite.


\subsection{The profinite Jordan--H\"older theorem}

A \emph{composition series} for a second countable, profinite group $G$ is a countable descending subnormal series $G = G_0 \triangleright G_1 \triangleright G_2 \triangleright \cdots$  consisting of closed subgroups
such that $\bigcap_{i \geq 0} G_i = \{1\}$ and such that each \emph{composition factor} $G_{i-1}/G_i$ is simple. The number of times that a composition factor appears, up to isomorphism, is called its \emph{multiplicity}.
The multiplicity can be finite or countably infinite.
It is well-known that every profinite group has a neighbourhood basis consisting of open, normal subgroups. It is also well-known that closed subgroups and Hausdorff quotients of profinite groups are again profinite. These two facts together imply that the $G_i$ are open subgroups of $G$ and the composition factors are finite. By \cite[Lemma 0.3.1(h)]{Wilson1998} the $G_i$'s form a neighbourhood basis of the identity.

For finite groups, composition series and composition factors are intimately tied with the classical Jordan--H\"older theorem. An analogue of this theorem holds for profinite groups. This is well-known to experts, but due to the lack of a suitable reference we give a proof here.

\begin{remark}
The only published mention the authors found of the Jordan--H\"older theorem for profinite groups was \cite[Section 2.2]{KlopschVannacci2017}, but without proof. The authors are grateful to Benjamin Klopsch for pointing out to us the argument presented here and to Colin Reid for some consultation concerning the theorem. For a slightly different proof, see his \texttt{mathoverflow} post \cite{Reid2014}.
\end{remark}

A {\em refinement} of a subnormal series is a series that contains each subgroup of the original series. In particular, the series $G \triangleright \{1\}$ and $G \triangleright G \triangleright  \{1\}$ are refinements of each other.

\begin{theorem}[Jordan--H\"older Theorem for profinite groups]\label{thm:JHprofinite}
Let $G$ be a profinite group.
\begin{enumerate}
    \item Every descending subnormal series of $G$ can be refined into a composition series. In particular, $G$ has a composition series.
    \item If $G$ is profinite, then any two composition series have, up to isomorphism and permutation, the same composition factors appearing with the same (finite or countably infinite) multiplicity.
\end{enumerate}
\end{theorem}

\begin{proof}
The first part can be proven verbatim as in the finite case, together with the classical fact that $G$ has a neighbourhood basis of the identity consisting of open, normal subgroups.

For the second part, let $G=G_0 \triangleright G_1 \triangleright G_2 \triangleright \cdots$ and $G=H_0 \triangleright H_1 \triangleright H_2 \triangleright \cdots$ be two composition series for $G$.
Let $A$ be a finite simple group. Let $n_1, n_2 \in \mathbb{N} \cup \{\infty \}$ be the multiplicities of $A$ in the first and the second composition series, respectively.
By symmetry, it is enough to show that $n_2 \geq n_1$. Let $n \in \NN$ with $n \leq n_1$.
Choose $k \geq 0$ such that $A$ appears at least $n$ times as quotient in the series $G_0 \triangleright \dots \triangleright G_k$.
Recall that since $G_k$ is open, it has finite index in $G$, so it has finitely many $G$-conjugates and $N= \bigcap_{g \in G} g^{-1}G_{k} g$ is an open, normal subgroup of $G$.
Now recall that the $H_i$ form a neighbourhood basis of the identity, so there exists $\ell \geq 1$ such that $H_\ell \leq N$.
Now setting $N' = \bigcap_{g \in G} g^{-1}H_{\ell} g$ we are getting two finite subnormal series
\begin{align*}
    G &= G_0 \triangleright G_1 \triangleright G_2 \triangleright \dots \triangleright G_k \triangleright N \triangleright H_1 \cap N \triangleright \dots \triangleright H_\ell \cap N = H_\ell \triangleright N' \\
    G &= H_0 \triangleright H_1 \triangleright H_2 \triangleright \dots \triangleright  
         H_\ell \triangleright N'.
\end{align*}
We can refine those subnormal series so that they have simple subquotients.
Let $n'$ and $n''$ be the multiplicities of $A$ in composition series for the finite groups $G/N'$ and $H_\ell/N'$, respectively.
By the Jordan--H\"older theorem for finite groups and the third isomorphism theorem, $A$ has to appear $n'$ times as a quotient in the refinement of $G = H_0 \triangleright H_1 \triangleright H_2 \triangleright \dots \triangleright  
         H_\ell \triangleright N'$ and $n''$ times in the refinement of $H_\ell \triangleright N'$.
But note that $n_2 \geq n'-n'' \geq n$. Since $n \leq n_1$ was arbitrary, we are done.
\end{proof}

By the above theorem the following concept is well-defined.

\begin{definition}
  Let $G$ be a profinite group. A finite, simple group is a \emph{composition factor} of $G$  with \emph{multiplicity} $n \in \NN \cup \{\infty\}$ if it is a composition factor with multiplicity $n$ in one, and hence every, composition series of $G$.
\end{definition}

\subsection{The local simple content}

Let $G$ be a totally disconnected, locally compact group.
We are interested in the simple groups that appear as composition factors of every open subgroups of $G$. The following definition is inspired by Gl\"ockner's concept of the local prime content, see \cite{Glockner2006}, and the work of Caprace, Reid and Willis in \cite{CapraceReidWillis2017a}.

\begin{definition}
  The \emph{local simple content} of a second countable, totally disconnected, locally compact group $G$ is the set of finite, simple groups (up to isomorphism) that are a composition factor of every compact, open subgroup of $G$.
\end{definition}

The following lemma allows us to detect the local simple content by only looking at one composition series of one compact, open subgroup.

\begin{lemma}\label{lem:local simple content equivalences}
Let $A$ be a finite simple group and let $G$ be a second countable, \tdlc group. The following are equivalent.
\begin{enumerate}
    \item The group $A$ is in the local simple content of $G$.
    \item There exists a compact, open subgroup $U \leq G$ such that $A$ is a composition factor with infinite multiplicity in $U$.
    \item For every compact open subgroup $U \leq G$, the group $A$ is a composition factor with infinite multiplicity in $U$.
\end{enumerate}
\end{lemma}

\begin{proof}
First we prove that (1) implies (3).
Assume that $A$ is in the local simple content of $G$. Let $U$ be a compact open subgroup of $G$ and let $U = U_0 \triangleright U_1 \triangleright \cdots$ be a composition series for $U$. Then, for every $k \geq 0$ the series $U_k \triangleright U_{k+1} \triangleright U_{k+2} \triangleright \cdots$ is a composition series for $U_k$ and $A$ has to appear as composition factor. Consequently, $A$ appears as composition factor of $U$ infinitely often.

It is obvious that (3) implies (2).

To show that (2) implies (1), assume that $U$ is a compact, open subgroup of $G$ and $U = U_0 \triangleright U_1 \triangleright \cdots$ is a composition series for $U$ such that $A$ appears as composition factor infinitely often. Let $V$ be any compact open subgroup of $G$ and let $V=V_1 \triangleright V_2 \triangleright \cdots$ be a composition series for $V$. Since the $V_i$ form a neighbourhood basis of the identity, there exists $m_0 \geq 1$ with $V_{m_0} \leq U_{0}$. Now the series
\[
V=V_0 \triangleright \dots \triangleright V_{m_0} = V_{m_0} \cap U_{0} \triangleright V_{m_0} \cap U_{1} \triangleright V_{m_0} \cap U_{2} \triangleright \cdots
\]
is a new composition series for $V$. By the Jordan--H\"older theorem for profinite groups, it has the same composition factors as the original series for $V$. But the $U_i$ form a neighbourhood basis of the identity, so for large $n$ the equality $V_{m_0} \cap U_{n} = U_{n}$ holds. Since $A$ appears as quotient of the subnormal series $U_{n} \triangleright U_{n+1} \triangleright \cdots$, we see that $A$ is a composition factor of $V$.

\end{proof}

The following considerations rely heavily on the proof of \cite[Proposition 4.6]{CapraceReidWillis2017a}.
A \emph{subquotient} of a group $H$ is a quotient of a subgroup of $H$.
Let $G$ be a \cgtdlc group and let $U \leq G$ be a compact, open subgroup.
Let $\Gamma$ be a Cayley--Abels graph for $G$ such that $U=G_\alpha$ for some vertex $\alpha \in \V\Gamma$.
We can use $\Gamma$ to produce a subnormal series for $U$.
Namely, take subgraphs $\Gamma_0 \subseteq \Gamma_1 \subseteq \Gamma_2 \subseteq \cdots$ of $\Gamma$ as follows:
\begin{enumerate}
\item The subgraph $\Gamma_0$ consists only of the vertex $\alpha$.
\item  For $i \geq 1$, choose a vertices $\alpha_i \in \V\Gamma_{i-1}$ satisfying the following two conditions.
\begin{itemize}
    \item The set of neighbours of $\alpha_i$ is not contained in $\V\Gamma_{i-1}$, and
    \item $\bigcup_{i \geq 0} \V\Gamma_i = \V\Gamma$. 
\end{itemize}
\item Now define $\Gamma_i$ to be the subgraph of $\Gamma$ induced by $\Gamma_{i-1}$ and $\Gamma(\alpha_i)$.
\end{enumerate}
The condition $\bigcup_{i \geq 0} \V\Gamma_i = \V\Gamma$ can, for example, be achieved by requiring that all vertices which have a strictly smaller distance to $\alpha$ are already contained in $\Gamma_{i-1}$.

Write $G_i = G_{(\Gamma_i)}$. The sequence
\[
U = G_{0} \triangleright G_{1} \triangleright G_{2} \triangleright \cdots
\]
is a descending subnormal series for $U$. Note that there could be repetitions in the series.
The condition
$\bigcup_{i \geq 0} \V\Gamma_i = \V\Gamma$ implies that
$\bigcap_{i \geq 0} G_{i} = K$, where $K$ is the kernel of the action of $G$ on $\Gamma$.
Note that $G_{i-1}$ permutes those neighbours of $\alpha_i$ that are not already contained in $\Gamma_{i-1}$ and the kernel of this action is $G_{i}$, so for $i \geq 2$ the group $G_{i-1}/G_{i}$ is a subgroup of the symmetric group $S_{d-1}$.
More precisely, let $\beta_i \in \V\Gamma_{i-1}$ be a neighbour of $\alpha_i$, then $G_{i-1}/G_{i}$ is a subquotient of the stabilizer of $\beta_i$ in the local action of $G$ on $\Gamma$.
Quotienting every subgroup in the given subnormal series for $U$ by $K$ gives a subnormal series for $U/K$, which by the Jordan--H\"older theorem for profinite groups can be refined into a composition series for $U/K$. This provides, together with Lemma \ref{lem:local simple content equivalences}, a proof of the following theorem.

%
%
%
%

\begin{theorem}{\rm(Cf.~\cite[Proposition 4.6]{CapraceReidWillis2017a})}
\label{thm:localsimplecontent}
Let $G$ be a compactly generated, totally disconnected, locally compact group and let $\Gamma$ be a Cayley--Abels graph for $G$ of degree $d$.  Define  $K$ as the kernel of the action of $G$ on $\Gamma$. Let $L \leq S_d$ be the local action of $G$ on $\Gamma$.

Suppose $A$ is an element of the local simple content of $G/K$.
Then $A$ is a subquotient of a point stabilizer in $L$, i.e. there exists $i \in \{1,\dots,d\}$ and $N \triangleleft H \leq L_i$ with $H/N \cong A$.
In particular, $H/N$ is a subquotient of $S_{d-1}$.
\end{theorem}

This theorem can be applied to the automorphism group of a regular tree.

\begin{corollary}\label{cor:regular tree}
Let $d \geq 2$ be a positive integer and assume $d \neq 5$.
Let $T_d$ be the $d$-regular tree.
Then $\md(\Aut(T_d))=d$ and $T_d$ is a minimal Cayley--Abels graph for $\Aut(T_d)$.
\end{corollary}


\begin{proof}
First it is clear that $T_d$ indeed is a Cayley--Abels graph for $\aut(T_d)$.

For $d=2$ we refer to Theorem \ref{T2ends}.

Let now $d \geq 3$, but $d \neq 5$. We can use the method described before Theorem \ref{thm:localsimplecontent} with $\Gamma = T_d$ to obtain a subnormal series for a vertex stabilizer.  The quotients of this subnormal series are all, except the first, isomorphic to $S_{d-1}$, which has the composition factors $A_{d-1}$ and $\mathbb{Z}/2\mathbb{Z}$.
By Lemma \ref{lem:local simple content equivalences} the local simple content of $\aut(T_d)$ consists of $A_{d-1}$ and $\mathbb{Z}/2\mathbb{Z}$.
Recall, or prove as exercise, that $d-1$ is the minimal number $k$ such that $A_{d-1}$ is a subquotient of $S_k$.
Recall as well, or prove using facts from Section \ref{sec:5-valent}, that $\aut(T_d)$ does not have any compact, normal subgroups.
Now Theorem \ref{thm:localsimplecontent} concludes the proof.
\end{proof}

This argument fails for the $5$-regular tree, because $A_4$ is not simple. In fact $\Aut(T_4)$ and $\Aut(T_5)$ have the same local simple content.
The result also holds for $\Aut(T_5)$, but the proof turned out to be surprisingly tricky and is given in Section~\ref{sec:5-valent}.
The proof above also provides an alternative proof for Part 1 in Corollary~\ref{cor:end stabilizer in tree} for $d \neq 5$.

\begin{remark}
In Theorem \ref{thm:localsimplecontent} we can not replace \lq\lq subquotient\rq\rq\ by \lq\lq subgroup\rq\rq. The reason is simply that a simple subquotient of a finite, simple group might not be isomorphic to a subgroup.
Concrete examples can be found, for example, among sporadic simple groups.
The McLaughlin group $\mathrm{McL}$ is a subquotient of the Conway group $\mathrm{Co_3}$.
In \cite{Mazurov1988} Mazurov gives for each sporadic finite simple group the minimal number of points on which it admits a non-trivial action.
The group $\mathrm{Co_3}$ acts non-trivially on 270 points, but $\mathrm{McL}$ cannot act non-trivially on less than 275 points.
Then, $\mathrm{McL}$ can also not be a subgroup of $\mathrm{Co_3}$.

To turn this into an actual counterexample to Theorem  \ref{thm:localsimplecontent} with \lq\lq subgroups\rq\rq\  instead of \lq\lq subquotients\rq\rq,
Let $N \triangleleft C \leq S_{270}$ satisfy $C/N \cong \mathrm{McL}$.
It suffices to find a graph $\Gamma$ of degree $271$ and a vertex-transitive subgroup $G \leq \aut(\Gamma)$ such that $G_{i-1}/G_i \cong C$ for infinitely many $i$.
Then $\mathrm{McL}$ is contained in the local simple content of $G$, but it is not contained in $S_{270}$.

For a concrete example, we can use Burger--Mozes universal groups $U(F)$ acting on trees.
The interested reader can find the definitions and basic properties in \cite{GarridoGlasnerTornier2018}.
It is not hard to see that the local simple content of $U(F)$ is the set of all composition factors of point stabilizers in $F$.
Taking $F \cong C \times \{1\}$, by the above $T_{271}$ is a $271$-regular Cayley--Abels graph for $U(F)$.
\end{remark}

\subsection{The local prime content and the scale function}

In this section we apply Theorem \ref{thm:localsimplecontent} to other invariants.
  
\begin{definition}{\rm (\cite[Definition 6.1]{Glockner2006})}
  The \emph{local prime content} of a totally disconnected, locally compact group $G$ is the set of all prime numbers $p$ such that every compact, open subgroup $U$ of $G$ contains a compact open subgroup $V \leq U$ with $p \mid \left|U:V\right|$.
\end{definition}

The following lemma gives the connection between the local prime content and the local simple content. If $G$ is not second countable, we do not know whether \ref{thm:JHprofinite}(2) still holds (see \cite{Reid2014}), so in the definition of composition factors, the ``and hence every''-part needs to be left out. This subtlety will however not be relevant for us.

\begin{lemma}
Let $G$ be a totally disconnected, locally compact group.
Then, the local prime content of $G$ contains the set of all prime numbers dividing the order of an element of the local simple content.

Equality holds if, for every compact open subgroup $U \leq G$, every prime number in the local prime content divides the order of at most finitely many composition factors of $U$ (without counting multiplicities).
In particular, this is true if $G$ acts faithfully on a Cayley--Abels graph.
\end{lemma}

\begin{proof}
Assume that $p$ divides the order of a finite, simple group $A$ that is in the local simple content of $G$.
Let $U$ be a compact, open subgroup. Let $U = G_0 \triangleright G_1 \triangleright \dots$ be a composition series of $U$ and let $n \geq 1$ be such that $G_n/G_{n-1} = A$. Then $p$ divides the index $|U:G_{n-1}|$.

For the other direction, 
let $p$ be in the local prime content.
Let $U \leq G$ be a compact, open subgroup.
By definition, there exist compact, open subgroups $U = V_0 \geq V_1 \geq V_2 \geq \dots$ such that $p$ divides the index $|V_{i+1}:V_i|$ for all $i \geq 1$.
By replacing $V_i$ by $\bigcap_{g \in U} g^{-1} V_i g$ we can assume that $V_i$ is normal in $U$; note that above intersection is a finite intersection because $V_i$ has finite index in $U$, and $\bigcap_{i \in \NN} V_i = \{1\}$. By Theorem \ref{thm:JHprofinite}(1) this subnormal series can be refined to a composition series. The prime number $p$ divides infinitely many composition factors, so by assumption, these infinitely many composition factors fall into only finitely many isomorphism classes of finite simple groups. By Lemma \ref{lem:local simple content equivalences} one of them has to be in the local simple content (note that (1) implies (3) implies (2) is also true for groups that are not second countable). 

\end{proof}

\begin{example}
The group $\prod_{n \geq 1} A_n$, where $A_n$ denotes the alternating group on $n$ symbols, is compact and second countable. Its local simple content is empty, because each $A_n$ is a composition factor with multiplicity $1$, but its local prime content is the set of all primes.
\end{example}

This lemma allows us to draw the following conclusion from Theorem~\ref{thm:localsimplecontent}. 

\begin{corollary}[\cite{CapraceReidWillis2017a}, Proposition 4.6]\label{cor:localprimecontent}
Let $p$ be a prime number that is in the local prime content of $G/K$ for every compact, normal subgroup $K \triangleleft G$.
Then, $\md(G) \geq p+1$.
\end{corollary}

The connection between the scale function and the local prime content is given in the following lemma by Gl\"ockner.

\begin{lemma}{\rm (\cite[Proposition 6.2]{Glockner2006})}
\label{lem:prime factors of scale contained in lpc}
Let $G$ be a \tdlc group. Let $g \in G$ and let $p$ be a prime number dividing $s(g)$. Then $p$ is contained in the local prime content of $G$.
\end{lemma}

\begin{lemma}\label{lem:scale_quotient}
Let $G$ be a totally disconnected, locally compact group and $K$ a compact, normal subgroup of $G$.  Denote the scale function on $G$ with $s$ and the scale function on $G/K$ with $s_{G/K}$.
If $g \in G$, then $s_{G/K}(gK)=s(g)$.
\end{lemma}

\begin{proof}
Suppose $\Omega$ is a set on which $G$ acts transitively such that the stabilizers of points are compact open subgroups of $G$, for example $\Omega = G/U$ for some compact, open subgroup $U \leq G$.  Since $K$ is compact in the given topology on $G$ it is also compact in the permutation topology on $G$ constructed from the action on $\Omega$ and thus $K$ has finite orbits on $\Omega$, see discussion after Lemma \ref{lem:compact-cocompact}.  Since $K$ is normal, $G$ has an action on $\Omega/K$ and the stabilizers in $G$ of points in $\Omega/K$ are compact open subgroups of $G$.  Let $\alpha$ be a point in $\Omega/K$. 
Note that the homomorphism $G \to \Sym(\Omega/K)$ factors through $G/K$, thus $(\alpha (gK))(G/K)_\alpha = (\alpha g)G_\alpha$ for every $g \in G$.
Now, by applying  Proposition~\ref{prop:scale} to the action of $G/K$ on $\Omega/K$ and to the action of $G$ on $\Omega/K$, we get
$$s_{G/K}(gK)=\lim_{n\to\infty} |(\alpha (gK)^n)(G/K)_\alpha|^{1/n}
=\lim_{n\to\infty} |(\alpha g^n)G_\alpha|^{1/n}=s(g).$$
\end{proof}

  
We are now in a position to derive a lower bound for the minimal degree in terms of the scale function.

\begin{corollary}\label{cor:prime factor of scale function md}
Let $G$ be a \cgtdlc group.  If $p$ is the largest prime that occurs as a factor of any of the values $s(g)$ for $g\in G$, then $\md(G)\geq p+1$. 
\end{corollary}

\begin{proof}
Let $K$ be a compact normal subgroup of $G$.
By Lemma \ref{lem:scale_quotient}, the prime number $p$ is also the largest prime factor of any of the values in the image of the scale function of $G/K$.
Lemma \ref{lem:prime factors of scale contained in lpc} says that the prime $p$ is contained in the local prime content of $G/K$. The result now follows from Corollary \ref{cor:localprimecontent}.
\end{proof}

The following is a simple application of the above results.

\begin{corollary}\label{cor:p adic Lie group md}
Let $G$ be a compactly generated $p$-adic Lie-group.  If $G$ is not nearly discrete, then $\md(G)\geq p+1$.
\end{corollary}

\begin{proof}  Gl\"ockner and Willis have shown in \cite[Theorem 5.2]{GlocknerWillis2001} that for a $p$-adic Lie group the condition that it is not nearly discrete is equivalent to the condtion that the group is not uniscalar. 
From the work of Gl\"ockner, see \cite[Theorem 2.1]{GlocknerWillis2001}, we know that every value of the scale function is a non-negative power of $p$.  The result now follows from the last corollary.  
\end{proof}


\section{The automorphism group of a $5$-regular tree}\label{sec:5-valent}

In Section~\ref{sec:composition},  Corollary \ref{cor:regular tree}, it is shown that if $d\geq 2$ and $d\neq 5$ then  the minimal degree of a Cayley--Abels graph for the automorphism group of the $d$-regular tree is $d$. The $1$-regular tree consists only of a single edge, so its automorphism group is compact and $\md(\aut(T_1))=0$.  What is left is to determine the minimal degree of $\aut(T_5)$.

\begin{theorem}\label{T5-regular}
Let $T_5$ be the $5$-regular tree. Then $\md(\aut(T_5)) = 5$.
In particular, $T_5$ is a Cayley--Abels graph of minimal degree for $\aut(T_5)$.
\end{theorem}

As a preparation for the proof of this theorem, we recall a few well-known facts about groups acting on trees.
First we remind the reader of the classical fact due to Tits \cite[Proposition~3.2]{Tits1970} that there are three types of automorphisms of a tree.

\begin{proposition}\label{prop:classification_Tits}
If $g$ is an automorphism of a tree $T$, then exactly one of the following is true:
\begin{enumerate}
    \item $g$ fixes some vertex $\alpha$;
    \item $g$ leaves some edge $\{\alpha, \beta\}$ invariant and transposes the vertices $\alpha$ and $\beta$;
    \item there is a $2$-way infinite arc $L$ in $T$ invariant under $g$ and $g$ induces a non-trivial translation on $L$.
\end{enumerate}
\end{proposition}

For an automorphism $g$ of a tree $T$ we denote by $\V T^g$ the set of vertices that are fixed by $g$. It is easy to see that in Case (1) of Proposition \ref{prop:classification_Tits} the set $\V T^g$ is the vertex set of a subtree of $T$.  

A group $G$ is said to act on on a tree $T$ {\em without inversion} if none of its elements satisfy Case (2).  If we replace $T$ with the first barycentric division of the tree 
(e.g. we replace $T_5$ by the bi-regular tree $T_{5,2}$), we can be sure that we have an action without inversion.  

In Case (3) $L$ is called the {\em translation axis} of $g$.  If $\alpha$ is a vertex in $L$ then the unique $s$-arc
$(\alpha_0, \ldots, \alpha_s)$ with $\alpha_0=\alpha$ and $\alpha_s=\alpha g$ is an $s$-arc in $L$. The integer $s$ is called the \emph{translation length} of $g$.


\begin{lemma}[\cite{Serre2003}, Proposition 26 (incl.\ proof)]
\label{lem:product of elliptic tree automorphisms}
  Let $T$ be a tree and let $g$ and $h$ be automorphisms of $T$.  Suppose $\V T^g\neq \emptyset$ and $\V T^h\neq \emptyset$.
  Then $gh$ fixes a vertex if and only if $\V T^g \cap \V T^h \neq \emptyset$.
  Otherwise $gh$ is a translation and the unique $s$-arc $(\alpha_0,\dots,\alpha_s)$ with $\alpha_0 \in \V T^g$ but $\alpha_1 \notin \V T^g$ and $\alpha_s \in \V T^h$ but $\alpha_{s-1} \notin \V T^h$ is contained in the translation axis of $gh$.
\end{lemma}

The following lemma is a consequence of the simple fact that a compact group acting continuously on a discrete set has finite orbits.

\begin{lemma}\label{lem:compact_acts_on_trees}
  Let $T$ be a tree and $G \leq \aut(T)$ a compact subgroup. Then, there exists a vertex or an edge in $T$ that is stabilized by every element of $G$.
\end{lemma}

We now present a lemma that can be deduced from various \lq\lq rigidity\rq\rq\ results for automorphism groups of trees, e.g.~\cite[Corollary 4.8(c)]{BassLubotzky1994}, but for completeness we include a direct proof.  

\begin{lemma}\label{lem:uniqueness of trees}
Let $d \neq d'$ be non-negative integers, bigger or equal to 3. Then $T_{d'}$ cannot be a Cayley--Abels graph for $\aut(T_d)$.
\end{lemma}

\begin{proof}   Suppose $G=\aut(T_d)$ acts transitively on $T_{d'}$ and the stabilizers of vertices are compact open subgroups of $G$.  Recall (or prove as an exercise using Lemma \ref{lem:compact_acts_on_trees}) that $G$ does not have any non-trivial, compact, normal subgroups, so this action is faithful.
As a consequence of Lemma~\ref{lem:compact_acts_on_trees}, the automorphism group of an infinite, regular tree has two conjugacy classes of maximal compact subgroups. One conjugacy class consists of the stabilizers of vertices and the other conjugacy class consists of the stabilizers of edges.   A compact open subgroup of $G$ acting on $T_{d'}$ has to fix a vertex or stabilize an edge in $T_{d'}$. 
Thus for every vertex $\alpha \in \V T_d$ we can find a vertex $\alpha' \in \V T_{d'}$ with $G_\alpha \leq G_{\alpha'}$ or an edge $e'$ in $T_{d'}$ with $G_\alpha \leq G_{e'}$.
By maximality of $G_{\alpha}$, such an inclusion has to be an equality.

First we consider the case when $G_\alpha=G_{e'}$ for a vertex $\alpha$ in $T_d$ and an edge $e'$ in $T_{d'}$.
Let $\alpha'$ be an end vertex of $e'$. Similar considerations as in the previous paragraph imply that there exists an edge $e$ of $T_d$ with $G_{\alpha'}=G_e$.
From the orbit-stabilizer theorem we see that $|G_{e'}:G_{\alpha'}\cap G_{e'}| = |\alpha' G_{e'}| \leq 2$.
However $|G_\alpha:G_e \cap G_{\alpha}| \geq d$ and this is a contradiction.

We are left with the other case.
The above allows us to construct a bijective map $\varphi \colon \V T_d\to \V T_{d'}$ such that if $\alpha\in \V T_d$ then $G_{\varphi(\alpha)}=G_\alpha$.
In particular, for all vertices $\alpha,\beta \in \V T_d$ we have $|\beta G_\alpha| = |G_\alpha : G_\alpha \cap G_\beta| = |G_{\varphi(\alpha)} : G_{\varphi(\alpha)} \cap G_{\varphi(\beta)}| = |\varphi(\beta) G_{\varphi(\alpha)}|$.
Also note that, because $T_d$ is a tree, $|\gamma G_\alpha| \geq |\beta G_\alpha|$ whenever $\beta$ lies on the unique $s$-arc from $\gamma$ to $\alpha$; the same statement holds for vertices in $T_{d'}$. 
In particular, fixing $\alpha$ and varying $\gamma \neq \alpha$ in $T_d$, the cardinality $|\gamma G_\alpha|$ is minimal if and only if $\gamma$ is a neighbour of $\alpha$.
In $T_{d'}$ we can only say that amongst those vertices $\gamma' \neq \alpha'$ with $|\gamma' G_{\alpha'}|$ minimal there exists a neighbour of $\alpha'$.
These considerations imply that for every vertex $\alpha \in \V T_d$ the map $\varphi$ sends some neighbour $\beta$ of $\alpha$ to a neighbour of $\varphi(\alpha)$.
Note that every element of $\varphi(\beta) G_{\varphi(\alpha)}$ has to be a neighbour of $\varphi(\alpha)$.
But then $d=|\beta G_\alpha|=|\varphi(\beta) G_{\varphi(\alpha)}|$ shows that $\varphi$ sends all neighbours of $\alpha$ to neighbours of $\varphi(\alpha)$.
Hence $\varphi$ defines a graph homomorphism $T_d \to T_{d'}$. But a graph morphism between trees that is bijective on the vertices has to be an isomorphism. This concludes the proof. 
\end{proof}

The same proof works for automorphism groups of bi-regular trees and in particular for $\aut^+(T_d)$, the index $2$ subgroup of $\aut(T_d)$ fixing the bipartition on the vertices of $T_d$. Another way to describe $\aut^+(T_d)$ is to say that it the subgroup generated by the vertex stabilizers in $G$.  It is well-known that $\aut^+(T_d)$ is the only non-compact, proper, open subgroup of $\aut(T_d)$ and that it is simple, see \cite{Tits1970}.

\begin{proof}[Proof of Theorem \ref{T5-regular}]
Set $G=\aut(T_5)$.
In what follows we assume, seeking contradiction, that $\Gamma$ is a minimal degree Cayley--Abels graph for $G$ and that the degree of $\Gamma$ is less than $5$.  Fix a vertex $\alpha_0$ of $\Gamma$ and write $B = G_{\alpha_0}$.

\smallskip

{\em Claim 1.}  The degree of $\Gamma$ is $4$ and the local action is $S_3$. The subgroup $\aut^+(T_5)$ of $G$ can not have a Cayley--Abels graph of degree less than 4.

{\em Proof.} Recall that $G$ does not have any non-trivial, compact, normal subgroups.
The local simple content of $G$ consists of the cyclic groups of order $2$ and $3$, so by Theorem \ref{thm:localsimplecontent} the degree of $\Gamma$ is equal to $4$. The local simple content of $\aut^+(T_5)$ is the same as the local simple content of $G$ and thus the degree of any Cayley--Abels graph is at least 4.  By Theorem \ref{thm:localsimplecontent},
 the cyclic groups of order $2$ and $3$ are composition factors of subgroups of point stabilizers of the local action. This is only possible if the local action is $S_3$ or $S_4$. However,
 \cite[Proposition 3.1]{Trofimov2007} implies that if the local action is $S_4$, then $\Gamma$ has to be a tree. That contradicts Lemma \ref{lem:uniqueness of trees}.

\smallskip

Let $\beta_1, \beta_2, \beta_3, \beta_4$ denote the neighbours of $\alpha_0$ such that $\beta_4$ is the neighbour that is fixed by $B$.

\smallskip

{\em Claim 2.}    The group $G$ has two orbits on the edges of $\Gamma$.  For $i=1, 2, 3, 4$ there exists $g_i \in G$
with $(\alpha_0,\beta_i)g_i=(\beta_i,\alpha_0)$.
Furthermore it is possible to choose $g_1, g_2$ and $g_3$ so that there exists $b \in B$ with $g_{i+1}=b^{-1}g_{i}b$ for $i=1,2$.

{\em Proof.} Choose an element $g_4\in G$ such that $\alpha_0 g=\beta_4$.  It is well-known that $G$ is unimodular and thus, by Lemma \ref{LSuborbits}, we see that $|\alpha_0 G_{\alpha_0 g_4}|=|(\alpha_0 g_4)G_{\alpha_0}|=|\beta_4 G_{\alpha_0}|=1$.  This implies that $\{\alpha_0,\beta_4\}$ is fixed by $G_{\beta_4}$. But also $\{\alpha_0,\beta_4 \}g_4$ is fixed by $g_4^{-1} G_{\alpha_0} g_4 = G_{\beta_4}$. Since there is only one fixed point in the local action, we get $\{\alpha_0,\beta_4 \}g_4 = \{\alpha_0,\beta_4\}$.

From this we also see that $G$ has two orbits on the edges, because if there did exist an element $g\in G$ such that $\{\alpha_0, \beta_4\}g=\{\alpha_0, \beta_1\}$ then either $g$ or $g_4g$ would fix $\alpha_0$ and take $\beta_4$ to $\beta_1$ contrary to assumptions.  We will say that the edges in the orbit $\{\alpha_0, \beta_4\}G$ are {\em red} and the edges in the other orbit $\{\alpha_0, \beta_1\}G$ are {\em blue}.    

Let now $h_1 \in G$ be such that $\alpha_0 h_1 = \beta_1$. Because the edge $\{\alpha_0, \beta_4\}h_1=\{\beta_1,\beta_4 h_1\}$ is fixed by $h_1^{-1}G_{\alpha_0}h_1=G_{\beta_1}$, the edge $\{\alpha_0, \beta_1\}h_1=\{\beta_1,\beta_1 h_1\}$ is not fixed by $G_{\beta_1}$. So there exists $h \in G_{\beta_1}$ with $\beta_1 h_1 h = \alpha_0$. We can now set $g_1 = h_1 h$, this element will satisfy $(\alpha_0,\beta_1)g_1 = (\beta_1,\alpha_0)$.

By Claim 1 there exists $b \in B$ with $\beta_1 b = \beta_2$ and $\beta_1 b^2 = \beta_3$. Clearly $g_i = b^{-i} g_1 b^i$ satisfies $(\alpha_0,\beta_i)g_i= (\alpha_0 b^{-i} g_1 b^i, \beta_i b^{-i} g_1 b^i) =(\beta_i,\alpha_0)$ for $i=1,2$.

\smallskip

{\em Claim 3.}    Write $A=\{g_1, g_2, g_3, g_4\}$ and $A'=\{g_1, g_2, g_3\}$.  Then $BA'B=BA'=A'B$.  It also follows that $Bg_4 B=Bg_4=g_4 B$.  Furthermore $G=\langle A, B\rangle=\langle g_1, g_4, B\rangle$, and the groups $\langle A', B\rangle=\langle g_1,B\rangle$ and $\langle g_4,B\rangle$ are compact open subgroups of $G$.

\smallskip

{\em Proof.}   Note that the set of vertices $\{\beta_1, \beta_2, \beta_3\}$ is invariant under $B$.
We show that $BA' = \{g \in G \mid \alpha_0 g \in \{\beta_1,\beta_2,\beta_3\} \}$. Indeed, if $\alpha_0 g= \beta_i$ then $g g_i^{-1}$ fixes $\alpha_0$ and thus $gg_i^{-1} \in B$ and $g\in Bg_i\subseteq BA'$.
This implies $BA' B \subseteq BA'$, and the direction $BA' \subseteq BA'B$ holds because $B$ contains the identity.
Note that invariance of blue edges under $G$ in particular implies that $BA'=(BA')^{-1}$.
We know that $g_i^2 \in B$ for $i=1,2,3$ and therefore $g_i B = g_i^{-1} B$ and $(BA')^{-1}=A'^{-1}B = A'B$.

The same kind of argument can be used to show that $Bg_4 B=Bg_4$.

Since one can move the vertex $\alpha_0$ to any of its neighbours by using an element from $A$, a standard graph theoretical argument shows that the group $\langle A, B\rangle$ acts transitively on $\Gamma$. As  $\langle A, B\rangle$ contains the vertex stabilizer $B$, we see that $G=\langle A, B\rangle$.  Denote by $\Gamma'$ the connected component of the graph $(\V\Gamma, \{\alpha_0, \beta_1\}G)$ (the graph with the same vertex set as $\Gamma$ but the red edges have been removed) that contains $\alpha_0$. The group $C=\langle A', B\rangle$ is an open subgroup of $G$ and acts transitively on the vertices of $\Gamma'$.  Thus $\Gamma'$ is a Cayley--Abels graph for $C$, and it has degree $3$.  But $C$ is an open subgroup (contains the open subgroup $B$)  and if $C$ is not compact then $C=G$ or $C=\aut^+(T_5)$.  By Claim 1 neither of these groups can have a Cayley--Abels graph of degree 3.  Thus $C$ must be compact.  

Note that $\langle g_4, B\rangle$ leaves the edge $\{\alpha_0, \beta_4\}$ in $\Gamma$ invariant and is equal to its stabilizer.  Hence $\langle g_4, B\rangle$ is a compact open subgroup of $G$.  

\smallskip

The next step is to relate the action of $G$ on $\Gamma$ to the action of $G$ on $T_5$.  To ease the presentation we will replace $T_5$ with its barycentric subdivision $T=T_{5, 2}$ and consider the action of $G$ on $T$ instead of the action on $T_5$.  This has the benefit that elements in $G$ that act like inversions on $T_5$ will now fix a vertex.  An element in $G$ that acts like a translation on $T_5$ will also act like a translation on $T$, but an element that does not act like a translation on $T_5$ will always fix a vertex in $T$.  In particular every compact subgroup of $G$ fixes some vertex in $T$. 

\smallskip

{\em Claim 4.}   The groups $C=\langle g_1, B\rangle$ and $D=\langle g_4, B\rangle$ both fix a vertex in $T$.   

The elements $g_1$ and $g_4$ do not fix a common vertex in $T$ and therefore $g_1g_4$ acts like a translation on $T$. Similarly, $g_2g_4$ and $g_3g_4$ are also translations.

{\em Proof.}  The groups $C$ and $D$ are both compact by Claim 3 and thus each of them fixes a vertex in $T$.

Let $\gamma_1$ denote a vertex in $T$ that is fixed by $C$ and let $\gamma_4$ denote a vertex that is fixed by $D$.  Suppose both $g_1$ and $g_4$ fix some vertex $\gamma$ in $T$.  Then $g_1$ fixes every vertex on the unique $s_1$-arc $P_1$ from $\gamma_1$ to $\gamma$ and $g_4$ fixes every vertex on the unique $s_4$-arc $P_4$ from $\gamma_4$ to $\gamma$.  But $B$ is contained in both $C$ and $D$, hence fixes both $\gamma_1$ and $\gamma_4$ and thus fixes every vertex in the unique $s_{14}$-arc $P_{14}$ from $\gamma_1$ to $\gamma_4$.
Since $T$ is a tree, $P_1$, $P_4$ and $P_{14}$ have a common vertex and that common vertex is fixed by $g_1, g_4$ and the group $B$. It is thus fixed by $G=\langle g_1, g_4, B\rangle = \aut(T)$ and we reached a contradiction. From Lemma~\ref{lem:product of elliptic tree automorphisms} it now follows that $g_1g_4$ acts like a translation on $T$.  The statements about $g_2g_4$ and $g_3g_4$ follow by symmetry.

\smallskip

{\em Claim 5.}   For $i=1, 2, 3$,  let $L_i$ be the translation axis of $t_i=g_ig_4$ and let $F$ be the fixed tree of $B$.   Then $L_1\cap L_2\cap L_3 \cap F$ contains at least one arc.

{\em Proof.}  Let $F_i$ be the fixed tree of $g_i$ for $i=1, 2, 3,4$. By Lemma \ref{lem:compact_acts_on_trees} a tree automorphism generates a subgroup with non-compact closure if and only if it is a translation.
Since $C$ is compact, it has a non-empty fixed tree, which is contained in the intersection $F \cap F_1 \cap F_2 \cap F_3$. In particular, this intersection is non-empty and $F_1 \cup F_2 \cup F_3$ is a tree.
The same argument shows that $F \cap F_4$ is non-empty.  In the last claim we showed that $F_i$ and $F_4$ are disjoint for $i=1,2,3$.   
 Thus also the $s$-arc from $F_1 \cup F_2 \cup F_3$ to $F_4$ has positive length. This $s$-arc is contained in $F$ because $F$ is connected and both $F \cap F_1 \cap F_2 \cap F_3$ and $F \cap F_4$ are non-empty.   Note also that by Lemma~\ref{lem:product of elliptic tree automorphisms} the  orientation of this arc fits with the orientation of all the translation axes.
 
 \smallskip 

{\em Claim 6.}   There is a vertex $\gamma$ in $L_1\cap L_2\cap L_3\cap F$ such that $\gamma t_1, \gamma t_2$ and $\gamma t_3$ are all different.  

{\em Proof.}  Suppose $\gamma$ is a vertex in $L_1\cap L_2\cap L_3\cap F$  and that $\gamma t_1=\gamma t_2$.  We know from Claim 2 that there is an element $b\in B$ such that $g_2=b^{-1}g_1b$ and $g_3=b^{-2}g_1b^2$.  Then $\gamma b = \gamma$ implies
$\gamma g_1g_4=\gamma t_1=\gamma t_2=\gamma g_2g_4=\gamma b^{-1}g_1bg_4=\gamma g_1bg_4$ and thus $\gamma g_1 b = \gamma g_1$. Hence $\gamma t_3=\gamma g_3g_4=\gamma b^{-2}g_1b^2g_4=\gamma g_1b^2g_4=\gamma g_1g_4=\gamma t_1$.  This shows that
$\gamma t_1=\gamma t_2=\gamma t_3 \in L_1 \cap L_2 \cap L_3$ and we see that if any two of the vertices $\gamma t_1, \gamma t_2$ and $\gamma t_3$ are equal, then all three are equal. In particular, they are contained in $L_1 \cap L_2 \cap L_3$.

Next we prove that if $\gamma t_1 = \gamma t_2 = \gamma t_3$ then $\gamma \in F$.
Let $g\in B$.  By Claim 3 there are $g', g''\in B$ with $g_4g=g'g_4$ and $g_1g'=g''g_i$ for some $i\in \{1, 2, 3\}$.  Recall that $\gamma g''=\gamma$ by definition of $F$.  Hence 
$$\gamma t_1g=\gamma g_1g_4g=\gamma g_1g'g_4=\gamma g''g_ig_4=\gamma g_ig_4=\gamma t_i=\gamma t_1.$$
Thus $\gamma t_1$ is fixed by $B$ and therefore $\gamma t_1\in L_1\cap L_2\cap L_3\cap F$.

However $t_i$ is a translation and $F$ is finite. Thus there exists a vertex $\gamma \in L_1\cap L_2\cap L_3\cap F$ such that $\gamma t_1\nin L_1\cap L_2\cap L_3\cap F$. By the above $\gamma t_1, \gamma t_2$ and $\gamma t_3$ must all be different.

\smallskip

{\em Claim 7.} (See Figure \ref{fig:halfplanes}.)  Let $\gamma$ be a vertex in $L_1\cap L_2\cap L_3\cap F$ such that $\gamma t_1, \gamma t_2$ and $\gamma t_3$ are all different.  
There is a unique vertex $\gamma'$ such that $(\gamma',\gamma)$ is an arc in $L_1\cap L_2\cap L_3$. 
Removing the edge $\{\gamma,\gamma'\}$ would divide $T$ into two connected components (``half-trees''); let $H$ be the component containing $\gamma$.
Set $R = \{t_1, t_2, t_3\}$ and write
$R^k = \{s_1\cdots s_k\mid s_1,\ldots, s_k\in \{t_1,t_2, t_3\}\}$ for $k \geq 0$. Then for every $r \in R^k$ with $k \geq 1$ the vertex $\gamma r$ is contained in exactly one of $H t_1, H t_2$ and $H t_3$.
If $r = r' t_j$ with $r' \in R^{k-1}$, then $\gamma r \in H t_j$.

\begin{figure}
    \begin{center}
    \scalebox{0.85}{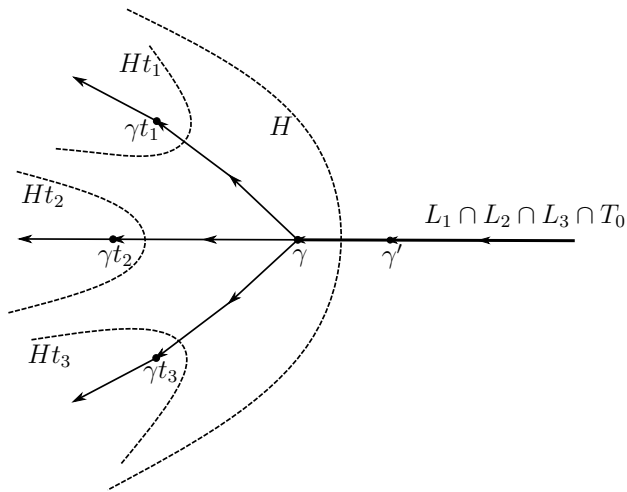}
    \end{center}
    \caption{The three translation axes splitting up}
     \label{fig:halfplanes}
\end{figure}

{\em Proof.} Recall that $L_1,L_2$ and $L_3$ are $2$-way infinite arcs. By abuse of notation we talk about the intersection $L_1 \cap L_2 \cap L_3$ as an $s$-arc for some $s \geq 1$.
Thus $\gamma r' \in H$.
 Since $\gamma t_1, \gamma t_2$ and $\gamma t_3$ are all different by assumption, the half-trees $H t_1$, $H t_2$ and $H t_3$ are all disjoint. For every $r' \in R^{k-1}$ and $i=1,2,3$ we have $H r' t_i \subseteq H t_i$. This proves the claim.
%


\smallskip

{\em Claim 8.}   For every $k\geq 0$ we have $\gamma R^k B = \gamma R^k$.

{\em Proof.} Note that $R^1 = A' g_4$. The equation $R^k B = B R^k$ follows directly from Claim 3. Now we are done since $\gamma \in F$.

\smallskip

{\em Conclusion of proof}.
Think of $T$ as a rooted tree with $\gamma$ as a root.  We will say that vertices $\alpha \neq \beta$ in $T$ are siblings if they have a common neighbour, and this common neighbour is closer to $\gamma$ than $\alpha$ and $\beta$. Since $B$ is open, there exists a finite subtree $F'$ of $T$ with $G_{(F')}\leq B$. It is possible to choose $F'$ such that the vertex set of $F'$ consists of a vertex $\delta$ and all vertices at distance at most $\ell$ from $\delta$ for some number $\ell$.  In particular for all vertices $\alpha \notin F'$ the orbit $\alpha B$ contains all the siblings of $\alpha$. The automorphisms $t_1, t_2, t_3$ are translations, so there is a number $k_0 \geq 0$ such that four vertices that are siblings of each other are contained in $\gamma R^{k_0}B = \gamma R^{k_0}$ (see Claim 8). We assume that $k_0$ is minimal.

 First we note that $k_0\neq 0$ because $R^0 = \{id\}$.  One can also exclude the possibility that $k_0=1$ since $\gamma R^1 = \{\gamma t_1, \gamma t_2, \gamma t_3\}$ has only three elements.
 Write $\gamma_1, \gamma_2, \gamma_3, \gamma_4\in \gamma R^{k_0}$ for the four siblings. Thus there are elements $p_1, p_2, p_3, p_4\in R^{k_0-1}$ and $s_1, s_2, s_3, s_4\in \{t_1, t_2, t_3\}$ such that $\gamma_i=\gamma p_i s_i$.    Since $k_0>1$ there is some $i\in \{1, 2, 3\}$ such that $\{\gamma_1, \gamma_2, \gamma_3, \gamma_4\}\subseteq H t_i$; here we are using that the $\gamma_i$ are siblings. Claim 7 implies that $s_1= s_2= s_3=s_4=t_i$.  Now $\gamma p_1, \gamma p_2, \gamma p_3, \gamma p_4$  are the images of the four siblings $\gamma_1, \gamma_2, \gamma_3, \gamma_4$ under the tree automorphism $t_i^{-1}$. In particular, because $k_0 \geq 2$, they are also siblings, and they are contained in $\gamma R^{k_0-1}$.  This means that $k_0$ was not minimal, and we have reached our final contradiction.

We have now shown that the assumption that $\aut(T_5)$ has a Cayley--Abels graph of degree less than 5 leads to a contradiction.  Thus $\md(\aut(T_5))=5$.  
\end{proof}

\bibliographystyle{abbrv}
\bibliography{references}

\end{document}

%% file: halfplanes.pdf_tex
\begingroup%
  \makeatletter%
  \providecommand\color[2][]{%
    \errmessage{(Inkscape) Color is used for the text in Inkscape, but the package 'color.sty' is not loaded}%
    \renewcommand\color[2][]{}%
  }%
  \providecommand\transparent[1]{%
    \errmessage{(Inkscape) Transparency is used (non-zero) for the text in Inkscape, but the package 'transparent.sty' is not loaded}%
    \renewcommand\transparent[1]{}%
  }%
  \providecommand\rotatebox[2]{#2}%
  \newcommand*\fsize{\dimexpr\f@size pt\relax}%
  \newcommand*\lineheight[1]{\fontsize{\fsize}{#1\fsize}\selectfont}%
  \ifx\svgwidth\undefined%
    \setlength{\unitlength}{251.66444434bp}%
    \ifx\svgscale\undefined%
      \relax%
    \else%
      \setlength{\unitlength}{\unitlength * \real{\svgscale}}%
    \fi%
  \else%
    \setlength{\unitlength}{\svgwidth}%
  \fi%
  \global\let\svgwidth\undefined%
  \global\let\svgscale\undefined%
  \makeatother%
  \begin{picture}(1,0.85009921)%
    \lineheight{1}%
    \setlength\tabcolsep{0pt}%
    \put(0,0){\includegraphics[width=\unitlength,page=1]{halfplanes.pdf}}%
    \put(0.4619342,0.62770837){\color[rgb]{0,0,0}\makebox(0,0)[lt]{\lineheight{1.25}\smash{\begin{tabular}[t]{l}$H$\end{tabular}}}}%
    \put(0.73497061,0.47274108){\color[rgb]{0,0,0}\makebox(0,0)[lt]{\lineheight{1.25}\smash{\begin{tabular}[t]{l}$L_1 \cap L_2 \cap L_3 \cap T_0$\end{tabular}}}}%
    \put(0.19621755,0.73881518){\color[rgb]{0,0,0}\makebox(0,0)[lt]{\lineheight{1.25}\smash{\begin{tabular}[t]{l}$Ht_1$\end{tabular}}}}%
    \put(0.50292851,0.40545901){\color[rgb]{0,0,0}\makebox(0,0)[lt]{\lineheight{1.25}\smash{\begin{tabular}[t]{l}$\gamma$\end{tabular}}}}%
    \put(0.66005035,0.40085654){\color[rgb]{0,0,0}\makebox(0,0)[lt]{\lineheight{1.25}\smash{\begin{tabular}[t]{l}$\gamma'$\end{tabular}}}}%
    \put(0.20657986,0.62552641){\color[rgb]{0,0,0}\makebox(0,0)[lt]{\lineheight{1.25}\smash{\begin{tabular}[t]{l}$\gamma t_1$\end{tabular}}}}%
    \put(0.03593101,0.22450491){\color[rgb]{0,0,0}\makebox(0,0)[lt]{\lineheight{1.25}\smash{\begin{tabular}[t]{l}$Ht_3$\end{tabular}}}}%
    \put(0.0239686,0.50721701){\color[rgb]{0,0,0}\makebox(0,0)[lt]{\lineheight{1.25}\smash{\begin{tabular}[t]{l}$Ht_2$\end{tabular}}}}%
    \put(0.23972381,0.1991467){\color[rgb]{0,0,0}\makebox(0,0)[lt]{\lineheight{1.25}\smash{\begin{tabular}[t]{l}$\gamma t_3$\end{tabular}}}}%
    \put(0.15971301,0.40285571){\color[rgb]{0,0,0}\makebox(0,0)[lt]{\lineheight{1.25}\smash{\begin{tabular}[t]{l}$\gamma t_2$\end{tabular}}}}%
  \end{picture}%
\endgroup%